\documentclass[a4paper,11pt]{article}
\usepackage[pagewise]{lineno}
\usepackage{ucs}
\usepackage[utf8x]{inputenc}
\usepackage{graphicx}
\usepackage{amsfonts}
\usepackage{amssymb}
\usepackage{amsmath}
\usepackage{amsthm}
\usepackage{enumerate}
\usepackage{fullpage}
\usepackage{ifthen}
\usepackage{epic}
\usepackage{authblk}
\usepackage{textcomp}
\usepackage[all]{xy}
\usepackage{mathrsfs}
\usepackage{subcaption}
\usepackage{titlesec}
\usepackage{soul}
\usepackage{stmaryrd}
\usepackage{etoolbox}
\makeatletter
\patchcmd{\ttlh@hang}{\parindent\z@}{\parindent\z@\leavevmode}{}{}
\patchcmd{\ttlh@hang}{\noindent}{}{}{}
\makeatother

\usepackage[hypertexnames=false,colorlinks=true,linkcolor=blue,citecolor=blue]{hyperref}
\usepackage[numbers,comma,square,sort&compress]{natbib}
\usepackage[letterpaper,text={7in,9in},centering]{geometry}

\usepackage{color}
\usepackage{titlesec}
\graphicspath{{eps/}{pdf/}}

\numberwithin{equation}{section}

\newtheorem{lem}{Lemma}[section]

\newtheorem{prop}[lem]{Proposition}

  {\begin{trivlist}\item[]{\bf Hypothesis #1 }\em}{\end{trivlist}}

  \newcommand{\R}{\mathbb{R}}
  \newcommand{\N}{\mathbb{N}}
  \newcommand{\T}{\mathbb{T}}
  \newcommand{\J}{\mathbf{J}}
  
\newcommand{\md}{\mathrm{d}}
\newcommand{\F}{\mathcal{F}}
\newcommand{\K}{\mathcal{K}}
\newcommand{\mL}{\mathcal{L}}

\newcommand{\mC}{\mathcal{C}}
\newcommand{\mS}{\mathcal{S}}
\newcommand{\mI}{\mathcal{I}}

\newcommand{\mk}{\mathbf{k}}
\newcommand{\mj}{j}
\newcommand{\ve}{\varepsilon}
\newcommand{\mz}{\mathbf{z}}
\newcommand{\mx}{\mathbf{x}}
\newcommand{\mh}{\mathbf{h}}

\newcommand{\my}{\mathbf{y}}

\newcommand{\V}{\mathcal{V}}
\newcommand{\W}{\mathcal{W}}

\newcommand{\ds}{\displaystyle}

\newcommand{\sinc}{\text{Sinc}}

\makeatletter\@addtoreset{figure}{section}\makeatother

\makeatletter \@addtoreset{equation}{section} \makeatother

\definecolor{darkblue}{rgb}{0.2,0.0,0.9}

 {\begin{trivlist} \item[]{\bf Proof. }}%
{\hspace*{\fill}$\rule{.4\baselineskip}{.4\baselineskip}$\end{trivlist}}

\newcommand{\B}{\mathcal{B}}

\usepackage{tikz}
\usetikzlibrary{shapes,arrows}

\title{Asymptotic preserving schemes for the FitzHugh-Nagumo transport equation with strong local interactions}
\author[1]{Joachim Crevat\footnote{\texttt{joachim.crevat@math.univ-toulouse.fr}}}
\author[1,2]{Francis Filbet\footnote{\texttt{francis.filbet@math.univ-toulouse.fr}}}
\affil[1]{Institut de Math\'ematiques de Toulouse ; UMR5219, Universit\'e de Toulouse ; UPS IMT, F-31062 Toulouse Cedex 9 France}
\affil[2]{Institut Universitaire de France}

\begin{document}
\maketitle

\begin{abstract}
  This paper is devoted to the numerical approximation of the
  spatially extended FitzHugh-Nagumo transport equation with strong
  local interactions based on  a particle method. In this regime, the
  time step can be subject to stability constraints related to the
  interaction kernel. To avoid this limitation, our approach is based
  on  higher-order  implicit-explicit numerical schemes. Thus, when
  the magnitude of the interactions becomes large,  this method
  provides a consistent discretization of the macroscopic
  reaction-diffusion  FitzHugh-Nagumo system.  We carry out some
  theoretical proofs and perform several numerical experiments that
  establish a solid validation of the method and its underlying
  concepts.
\end{abstract}


\section{Introduction}
\setcounter{equation}{0}


The FitzHugh–Nagumo (FHN) system \cite{FIT}, \cite{NAG}, models the pulse transmission in animal nerve axons and allows to describe complicated interactions of neurons in large neural networks. More precisely, we consider a  network composed of $n\in\mathbb{N}$ neurons interacting with each other, where each one is labeled by $i\in\{1,...,n\}$, and endowed with a parameter $\mx_i\in \R^d$ for $d\in\{1,2,3\}$ standing for the constant spatial position in the network. The FHN system accounts for the variations of the membrane potential $v_i$ of a neuron coupled to an auxiliary variable $w_i$ called the adaptation variable. It can be written as follows for all $i\in\{1,...,n\}$,
\begin{equation}\label{eq:FHN}
\left\{\begin{array}{l}
\dfrac{\md v_i }{\md t}\,=\,N(v_i) - w_i + \dfrac{1}{n\,\ve^2}\ds\underset{j=1}{\overset{n}{\sum}}\,\Psi_\ve(\|\mx_i-\mx_j\|)\,(v_j-v_i), \\ \, \\
\dfrac{\md w_i}{\md t} \,=\, \tau\,(v_i-\gamma\,w_i),
\end{array}\right.
\end{equation}
where $\tau \geq 0$ and $\gamma > 0$ are given constants,  $N(v)= v\,(1-v)\,(v-\theta)$ with $\theta\in(0,1)$ a fixed parameter whereas
$\ve>0$ is a scaling small parameter describing the intensity of  local interactions between neurons. For all $\ve>0$ and for all $\mx\in\R^d$, the  connectivity kernel $\Psi$ only depends on the relative distance between neurons and is given by 
$$
\Psi_\ve(\|\my\|)\,:=\,\dfrac{1}{\ve^d}\,\Psi\left(
  \dfrac{\|\my\|}{\ve} \right), \quad \my\in\R^d,
$$
where $\Psi:\R^+\rightarrow\R^+$. This scaling with respect to $\ve$
means that when $\ve$ goes to zero, space interactions are highly dominated by local ones
compared to long range correlations.  In the rest of this article, we assume that the  connectivity kernel $\Psi$ is nonnegative and rapidly vanishing at infinity, hence we introduce the following quantities,
\begin{equation}\label{hyp:psi}
  \left\{
    \begin{array}{l}
    \ds  \overline{\Psi}\,:=\,\ds\int_{\R^d}\Psi(\|\my\|)\,\md\my \,>\,0, \\[0.9em]
      \ds\overline{\sigma}\,:=\,\ds\frac{1}{2}\int_{\R^d}\Psi(\|\my\|)\,\|\my\|^2\md\my \,>\,0,
\end{array}\right.
      \end{equation}
which will play an important role later. A typical example for $\Psi$ is a Gaussian function, or the indicator function in a compact set.

In \cite{CREmf}, we proved that as the number of neurons $n$ goes to
infinity and for $\Psi \in \text{Lip}_b(\R^+)$, the set of neurons at time $t>0$ and position $\mx\in\R^d$  can be described  by a distribution function $f^\ve(t,\mx,.)$ solution of a mean-field equation,
\begin{equation}
  \label{eq:kin-1}
\left\{\begin{array}{l}
         \partial_t f^\ve \,+\,\partial_v\left[ f^\ve\,\left( N(v) \,-\, w \,+\,\K_\ve[f^\ve] \right) \right] \,+\,\partial_w\left[ f^\ve\,A(v,w) \right] \,=\,0,
         \\ \, \\
         f^\ve(t=0,\mx,.) \,=\,  f_0^\ve(\mx,.),  
\end{array}\right.
\end{equation}
with  $\K_\ve[f^\ve]$ and $A$ given by
         \begin{equation}
  \label{eq:kin-2}
\left\{\begin{array}{l}
         \K_\ve[f^\ve](t,\mx,v)
         \,=\,\dfrac{1}{\ve^2}\ds\int_{\R^d}\int_{\R^2}\Psi_\ve\left(
         \|\mx-\mx'\|\right)\,(v'-v)\,f^\ve(t,\mx',\md v',\md w') \md\mx',
         \\ \, \\
A(v,w) \,=\, \tau\,\left( v \,-\, \gamma\,w \right).
\end{array}\right.
\end{equation}
Here, we want to  construct numerical solutions of \eqref{eq:kin-1}--\eqref{eq:kin-2} using particle methods, which consist in approximating the distribution function by a finite number of macro-particles. The trajectories of these particles are determined from the characteristic curves corresponding to the \eqref{eq:kin-1}. Indeed, for any initial data $f_0^\ve$  with finite second moments in $\mx\in\R^d$ and $(v,w)\in\R^{2}$, the solution of \eqref{eq:kin-1}--\eqref{eq:kin-2} is uniquely determined as the push-forward of $f_0^\ve$ by the flow of the characteristic system of equations associated to \eqref{eq:kin-1}--\eqref{eq:kin-2}, which can be written for $(t,\mx)\in\R^+\times\R^d$ and $(v,w)\in\R^2$ as
\begin{equation}
  \label{eq:char}
\left\{\begin{array}{l}
         \dfrac{\md \V^\ve }{\md t} \,=\, N\left(\V^\ve\right) \,-\, \W^\ve \,+\, \K_\ve[f^\ve]\left(t,\mx, \V^\ve\right), 
\\ \, \\
         \dfrac{\md \W^\ve}{\md t} \,=\, A\left( \V^\ve,\W^\ve \right) ,
\\ \, \\
\V^\ve(0) \,=\, v , \quad \W^\ve(0) \,=\, w.
\end{array}\right.
\end{equation}
Then  we denote by $\Phi_{t,\mx}$ the  flow $(v,w)\in\R^{2}
\mapsto \Phi_{t,\mx}(v,w) \,=\, (\V^\ve,\W^\ve)(t,\mx,v,w)$, with  $ \Phi_{t,\mx}(v,w) \in\R^2$, hence the solution {of} \eqref{eq:kin-1}--\eqref{eq:kin-2} is given by 
\begin{equation}
  \label{f:01}
f^\ve(t,\mx,.) \,\,=\,\,\Phi_{t,\mx}\#  f_0^\ve(\mx,.),
\end{equation}
{ where $\#$ is our notation for a push-forward,} that is, for any test-function $\varphi$ and $B\subset \R^2$,
$$
\int_B \varphi(v,w) \, f^\ve(t,\mx,\md v,\md w)
\,=\,\int_{\Phi_{t,\mx}^{-1}(B)} \varphi\circ \Phi_{t,\mx} \,\,
f_0^\ve(\mx,\md v,\md w).
 $$
We also define for all $(t,\mx)\in \R^+\times\R^d$ and $\ve>0$ the following macroscopic quantities,
\begin{equation}\label{def:macro}
\rho^\ve\,\begin{pmatrix} 1 \\ V^\ve \\ W^\ve\end{pmatrix}(t,\mx) \,:=\, \ds\int_{\R^2} \begin{pmatrix} 1 \\ v \\ w \end{pmatrix} \,f^\ve(t,\mx,\md v,\md w),
\end{equation}
so that $\rho^\ve(t,\mx)$ is the average neuron density in the network
at time $t$ and location $\mx$, and $(V^\ve,W^\ve)$ is the average
pair membrane potential - adaptation variable. Therefore, we observe
that $\K_\ve[f^\ve]$ may be written with respect to the macroscopic
quantities $\rho^\ve$ and $\rho^\ve\,V^\ve$ as
\begin{equation}
  \label{def:Kf}
\K_\ve[f^\ve](.,v)  =  \frac{1}{\ve^2}\,\left[
  \Psi_\ve\star(\rho^\ve V^\ve) \,-\,   \Psi_\ve\star\rho^\ve \, v\,\right]\,,
\end{equation}
where $\star$ denotes the standard convolution product in $\mx$.

{ An important issue in the numerical simulation of
  \eqref{eq:char} is that when the parameter $\ve$
  is small, the numerical error of a  classical time explicit scheme
  may become large. For instance with an explicit Euler scheme,  the
  error may behave as $O(\frac{\Delta t}{\ve^2})$, where $\Delta t$ is
  the time step, hence the scheme is
  not appropriate for $\ve \ll 1$. Here we want to design a numerical
  scheme which is less sensitive to this parameter $\ve>0$ in order to
  keep a control on the numerical error when $\ve \ll 1$.}

Before describing and analyzing a class of numerical methods for
\eqref{eq:kin-1}--\eqref{eq:kin-2}  in the presence of strong local
space interactions ($\ve\ll 1$), we first briefly expound what may be
expected from the continuous model in the limit $\ve\to 0$.

On the one hand, by integrating \eqref{eq:kin-1}  with respect to $(v,w)\in\R^2$, we observe that for all $t\geq 0$
$$
\rho^\ve(t,\mx) = \rho_0^\ve(\mx), \quad \mx\in\R^d
$$
and moreover we suppose that  it does not depend neither on $\ve$, so that $\rho^\ve(t,.)=\rho_0$ with
\begin{equation}
  \label{hyp:rho}
\rho_0\,\geq\,0,\quad\rho_0\in L^\infty(\R^d).
\end{equation}
On the other hand,  using (\ref{def:Kf}), we observe that
\begin{equation}
  \label{K2}
\int_{\R^2}  \K_\ve[f^\ve](t,\mx,v) \,f^\ve(t,\mx,\md v\,\md w)  \,=\, \frac{\rho_0}{\ve^2} \,\left[\, \Psi_\ve \star (\rho_0\,V^\ve) \,-\, (\Psi_\ve\star\rho_0)\, V^\ve\, \right]\,. 
\end{equation}
Hence,  multiplying \eqref{eq:kin-1}  by $v$ (resp. $w$) and
integrating with respect to $(v,w)\in\R^2$ and using \eqref{K2}, we get a time evolution equation for the macroscopic quantities $(\rho_0 V^\ve,\, \rho_0 W^\ve)$ as
\begin{equation}
  \label{eq:macroeps}
\left\{\begin{array}{l}
\ds        \partial_t (\rho_0V^\ve) -\frac{\rho_0}{\ve^2} \left[
         \Psi_\ve \star (\rho_0V^\ve) -
         (\Psi_\ve\star\rho_0) V^\ve \right] =
         \int_{\R^2}N(v) f^\ve(.,\md v,\md w) - \rho_0W^\ve,
 \\ \, \\
\partial_t (\rho_0\,W^\ve) =\rho_0A\left( V^\ve,W^\ve\right).
\end{array}\right.
\end{equation}
Of course, this system is not closed since the right hand side of the equation on
$\rho_0\,V^\ve$ again depends on the distribution function
$f^\ve$. However, in the regime of strong local interactions \cite{CRE}, that is,
in the limit $\ve\rightarrow 0$,  the singular term in $\ve^{-2}$
indicates that the distribution function $f^\ve$ converges towards a
Dirac distribution in $v$ centered in $V^\ve$. Then applying a Taylor
expansion of the solution $V^\ve$,  the right hand side of
\eqref{eq:macroeps}  gives rise to a diffusive operator for the
spatial interactions at zeroth order with respect to $\ve$. It yields that $(\rho_0\,V^\ve,\rho_0\,W^\ve)$ converges towards a limit pair $(\rho_0\,V,\rho_0\,W)$ satisfying the FHN reaction-diffusion system,
\begin{equation}\label{eq:macro}
\left\{\begin{array}{l}
         \rho_0\,\left( \,\partial_t V \,-\overline{\sigma}\left[\Delta\left( \rho_0\,V \right) \,-\, V \Delta\rho_0 \right]\,-\,N(V) \,+\, W\right) \,=\, 0,
         \\ \, \\
\rho_0\,\left( \,\partial_t W \,-\, A(V,W) \,\right) \,=\, 0,
\end{array}\right.
\end{equation}
where $\overline{\sigma}$ is defined in \eqref{hyp:psi}.  We refer to \cite{CRE}
for more details on this asymptotic analysis.


We now come to our main concern in the present article and seek after
a numerical method that is able to capture these expected asymptotic
properties, even when numerical discretization parameters are kept
independent of $\ve$ hence are not adapted to the stiffness degree of
the space interactions. Our objective enters in the general framework of so-called Asymptotic Preserving (AP) schemes, first introduced and widely studied for dissipative systems as in \cite{JIN}, \cite{KLA}. Yet, in opposition with collisional kinetic equations in hydrodynamic or
diffusion limits, transport equations like \eqref{eq:kin-1}
involve of course some stiffness in time but it is also crucial to
take care of the space discretization in order to capture the
correction terms of  the non-local operator
$\K_\ve[f^\ve]$. By many respects this makes the identification of suitable schemes much more challenging. 

{
One of the interest of the study of AP schemes is to numerically determine a rate of convergence of the transport equation \eqref{eq:kin-1} as the parameter $\ve$ goes to $0$. Thus, we can compare this numerical rate of convergence with what we derived in the continuous framework in \cite{CRE}.
}

In \cite{CREmf}, the author proposed a numerical approximation to \eqref{eq:kin-1}--\eqref{eq:kin-2} using a standard particle
method. However, as the parameter $\ve$ goes to $0$, that is when the
range of interactions between neurons shrinks and their amplitude
grows, the time step and spatial grid size have to tend to zero too,
hence the scheme cannot be consistent with the limit system
\eqref{eq:macro} in the limit $\ve\rightarrow 0$. In a
different context \cite{FIL},  \cite{FIL2},  F. Filbet \& L. M. Rodrigues developed a
particle method for  the Vlasov-Poisson system with a strong external
magnetic field, which is able  to capture accurately the
non stiff part of the evolution while allowing for coarse
discretization parameters. 

Here, we show how this approach may be extended to transport equations
like \eqref{eq:kin-1} to deal with the time discretization. However,
it is not sufficient since an appropriate space discretization
technique is mandatory to capture the diffusive operator in
\eqref{eq:macro} in the limit $\ve\rightarrow 0$. In \cite{BUE}, the
authors apply a spectral collocation method to provide numerical
approximations of reaction-diffusion equations, with fractional
spatial diffusion. Their method obviously can also be applied for
local diffusions as in the FitzHugh-Nagumo reaction-diffusion system
\eqref{eq:macro}. On the other hand, the spectral collocation method
also provides numerical approximations of differential equations with
integral terms. For example, in \cite{PAR}, \cite{FR}, \cite{FPM},
\cite{FHJ} and \cite{FPR},  the authors use 
fast spectral methods for the non-local  Boltzmann operator, which lead to compute the time evolution of Fourier
coefficients of the solution instead of the solution
itself. Therefore, this approach considerably simplifies the
computation of the integral collision term and may be applied in our
context.  Moreover, we will show that a suitable formulation allows to perform a Taylor expansion of the solution in the Fourier
space and to recover a consistent discretization of the macroscopic
system \eqref{eq:macro} in the limit $\ve \rightarrow 0$, which
guarantee the asymptotic preserving property. Finally, another difficulty in our framework is to prove the convergence when
$\ve$ vanishes of the nonlinear term in \eqref{eq:char} involving the
cubic function $N$. The idea to circumvent this
issue is to use, as in the continuous framework \cite{CRE}, the stiff term in \eqref{eq:char}, which stands for
the interactions between neurons throughout the network to prove that the solution $f^\ve$ converges towards a Dirac mass in $v$, that is
all the membrane potential of the neurons at position $\mx$ are
synchronized. Thus, it is possible to identify the asymptotic  of the
nonlinear term in \eqref{eq:char}. We will show that the particle
approximation of the distribution in $(v,w)\in\R^2$ is particularly
well suited to achieve this.


The rest of the paper is organized as follows. In Section \ref{sec:particle}, we present the   particle method for
the transport  equation \eqref{eq:kin-1}--\eqref{eq:kin-2} and propose
an appropriante time discretization technique in order to preserve the
correct asymptotic when $\ve \ll 1$. Then, we provide first and second order
schemes and verify the consistency when $\ve$ tends to zero. Finally, in Section \ref{sec:num}, we present
some numerical simulations to illustrate our results, and to study the
dynamics of \eqref{eq:kin-1}--\eqref{eq:kin-2} with different different sets of parameters and different heterogeneous neuron densities.

\section{A numerical scheme for the FitzHugh-Nagumo transport
  equation}
\label{sec:particle}
\setcounter{equation}{0}

This section is devoted to the construction of the numerical schemes
for  \eqref{eq:kin-1}- \eqref{eq:kin-2}. We first focus on the
discretization of the nonlocal operator $\K_\ve[f^\ve]$ in \eqref{eq:kin-2},
for which we propose a spectral collocation method based on the
discrete fast Fourier method. Then, we treat the transport equation
\eqref{eq:kin-1} using a particle method for the microscopic variable
$(v,w)\in\R^2$ and provide  first and second order  semi-implicit
schemes for the time discretization. This algorithm is constructed in
order to get a consistent approximation in the limit $\ve\rightarrow 0$.

For sake of clarity, we drop the dependence with respect to $\ve$ on
the distribution function $f^\ve$ and on the non-local operator $\K_\ve[f^\ve]$.

\subsection{Computation of the Non-local operator}
We first look for an  approximation of the  operator $\K[f]$ given
in \eqref{eq:kin-2}. In view
  of applying a Fourier spectral method in space, we write
  $\K[f]$ as
  \begin{eqnarray*}
\K[f](t, \mx, v)  &=& \ds \frac{1}{\ve^2}\,\int_{\R^d}
  \Psi_\ve(\|\my\|)\,\rho_0(\mx-\my)\,\left( V(t,\mx-\my) -
                          v\right)\,\md\my\,.
  \end{eqnarray*}
  Then we define a truncated operator $\K^S[f]$  in the following way.
\begin{lem}
Suppose that ${\rm Supp}\left(\rho_0\right)\subset \mathcal{B}(0,S)$, where
$\mathcal{B}(0,S)$ is the ball of radius $S>0$ centered at the
origin and choose $\ve\in (0,1)$. Then, for any $(t,\mx, v,w)\in\R^+\times {\mathcal
  B}(0,S)\times \R^2$, $f$ is solution {of}
$$
\partial_t f\,+\,\partial_v\left[ f\,\left( N(v) \,-\, w
       \,+\,\K^S[f] \right) \right] \,+\,\partial_w\left[
     f\,A(v,w) \right] \,=\,0,
$$
where for any $(\mx, v)\in {\mathcal B}(0,S)\times \R$,
\begin{equation}
   \label{eq:trun-1}
\K^S[f](t,\mx , v)  = \frac{\chi_{{\mathcal B}(0,S)}}{\ve^2}(\mx)\,\int_{{\mathcal B}(0,2S)} \Psi_\ve(\|\my\|)\,\rho_0(\mx-\my)\,\left( V(t,\mx-\my) - v\right)\,{\md\my},
\end{equation}
where $\chi_{{\mathcal B}(0,S)}$ denotes the characteristic function in the ball ${\mathcal B}(0,S)$.
\end{lem}
\begin{proof}
On the one hand, since ${\rm Supp}\left(\rho_0\right)\subset\mathcal{B}(0,S)$ and for all $t\geq
0$, the density $\rho(t)=\rho_0$, we get that for any $\mx\in\R^d$, the transport equation \eqref{eq:kin-1} can be written as 
  $$
   \partial_t f \,+\,\partial_v\left[ f\,\left( N(v) \,-\, w
       \,+\,\chi_{{\mathcal B}(0,S)}\,\K[f] \right) \right] \,+\,\partial_w\left[
     f\,A(v,w) \right] \,=\,0\,.
   $$
    Then it is enough to consider only
  $\mx\in {\mathcal B}(0,S)$. On the other hand,  the domain of integration of the
operator $\K[f]$ is such that
$$
\|\my\| \,\leq\, \|\mx\| \,+\, \| \my-\mx\| \,\leq \, 2\,S,
$$
 hence for any $(t,\mx,v) \in \R^+\times{\mathcal B}(0,S)\times \R$,
 $$
\K[f](t,\mx,v) \,=\, \frac{1}{\ve^2}\,\int_{{\mathcal B}(0,2S)}
\Psi_\ve(\|\my\|)\,\rho_0(\mx-\my)\,\left( V(t,\mx-\my) -
  v\right)\,\md \my.
 $$
Thus, we define the truncated operator \eqref{eq:trun-1} as $\ve^2\,\K^S[f]  = \chi_{B(0,S)}\,\K[f]$.
\end{proof}

Actually the operator  $\K^S[f]$ can be
seen as convolution products between $(\rho_0,\rho_0 V)$ and the
connectivity kernel $\Psi_\ve$, that is,
$$
\K^S[f](t,\mx,v) \,=\, \frac{1}{\ve^2}\,\left(\mL^S[\rho_0V](t,\mx)  \,-\, v\,\mL^S[\rho_0](\mx) \right),
$$
where $\mL^S$ is given by
\begin{equation}
  \label{Lu}
\mL^S[u] \,=\,  \Psi_\ve\star u, \quad u\in\{\rho_0, \,\rho_0V\}. 
\end{equation}
In the sequel, we choose for simplicity $S=\pi/2$ such that
$\mathcal{B}(0,S)\subset \T:=[-\pi,\pi]^d$, and consider a set of equidistant points
$(\mx_\mj)_{\mj\in \J_{n_x}} \subset \T$ with $\J_{n_x}:=\llbracket -n_x/2 ,
n_x/2-1\rrbracket^d$ where $n_x$ is an even integer.  An efficient strategy to
approximate this nonlocal term is the spectral or spectral collocation
methods \cite{HES,PAR}. We suppose that the density $\rho_0$ and the macroscopic membrane
potential $V$  are both known at the mesh points $(\mx_\mj)_{\mj\in
  \J_{n_x}}$, then we compute an approximation of the Fourier
coefficients for $u\in\{\rho_0,\, \rho_0V\}$ as,
$$
\widehat{u}(t,\mk) \,:=\, \dfrac{1}{n_x^d}\,\ds\sum_{\mj\in \J_{n_x}}
u(t,\mx_\mj)\,e^{-i\,\mk\cdot\mx_\mj}, \quad \mk\in\J_{n_x}\,.
$$
and get a trigonometric polynomial
$$
u_{n_x}(t,\mx)\,:=\,\ds\underset{\mk\in
  \J_{n_x}}{\sum}\,\widehat{u}(t,\mk)\,e^{i\,\mk\cdot\mx}, \quad u\in\{\rho_0,\, \rho_0V\}\,.
$$
Therefore, we substitute this polynomials in \eqref{Lu}, which yields a discrete operator $\mL_{n_x}^S$  given by
 \begin{equation}
  \label{def:L0}
\mL_{n_x}^S[u] \,:=\, \sum_{\mk\in \J_{n_x}} \widehat{\mL}^S[u](t,\mk)\, e^{i\,\mk\cdot\mx},
\end{equation}
where  $\widehat{\mL}^S[u]$ is given by
\begin{equation}
  \label{def:hatL0}
\widehat{\mL}^S[u](t,\mk) \,\,=\,\,  (2\pi)^d
\,\widehat\Psi_\ve(\mk) \,\widehat{u}(t,\mk)
\end{equation}
and $\widehat\Psi_\ve$ is the expansion coefficient depending on the
connectivity kernel
\begin{equation}
\label{def:hat1}
\widehat\Psi_\ve(\mk)\,=\,\frac{1}{(2\pi)^d}\int_\T \Psi_\ve(\|\mx\|)\,e^{-i\,\mk\cdot\mx}\,\md \mx.
\end{equation}

Finally the approximation $\K_{n_x}^S[f]$ of  the operator $\K^S[f]$
is provided  by 
  \begin{equation}
\label{def:Ksn}
\K_{n_x}^S[f](t,\mx,v) \,=\, \frac{1}{\ve^2}\,\left(\mL^S_{n_x}[\rho_0V](t,\mx)  \,-\, v\,\mL^S_{n_x}[\rho_0](\mx) \right).
\end{equation}
  
Let us focus on the computation of the kernel modes
$(\widehat\Psi_\ve(\mk))_{\mk\in \J_{n_x}}$ for any fixed parameter
$\ve>0$. In the spirit of \cite{PAR} for the Boltzmann equation, our
purpose is to prove that these coefficients can be computed as
one-dimensional integrals, so that we can store them in an array, but
also to compute the asymptotic limit when $\ve\rightarrow 0$ in order to
ensure that the scheme is consistent and stable when $\ve\ll 1$. 

Using the change of variable $\mx = r\,\omega$, for $r\geq 0$ and $\omega\in\mathbb{S}^{d-1}$, 
we get:
$$
\widehat\Psi_\ve(\mk) \,=\, \dfrac{1}{(2\pi)^d}\ds\int_0^\pi \Psi_\ve(r) \, r^{d-1} \,{I}(\mk,r)\,\md r,$$
where
$${I}(\mk,r) \,:=\,
\ds\int_{\mathbb{S}^{d-1}}\exp(-i\,r\,\mk\cdot\omega)\,\md\omega.$$
Then,  changing the variable $r$ into $s=r/\ve$, we get
$$
\widehat\Psi_\ve(\mk) \,=\,  \dfrac{1}{(2\pi)^d}\ds\int_0^{\pi/\ve} \Psi(s) \, s^{d-1} {I}(\mk,\ve\,s)\,\md s.
$$

To complete the computation of the function ${I}$, we have to
study separately each possible value for the spatial dimension
$d\in\{1,2,3\}$. 
\paragraph{One-dimensional case: $d=1$. } Since
$\mathbb{S}^0\,=\,\{-1,1\}$, it is straightforward to check that for
any $\mk\in \J_{n_x}$,
$${I}(\mk,r)\,=\,2\,\cos(r\,|\mk|),$$
hence  we get:
$$
\widehat\Psi_\ve(\mk) \,=\,\dfrac{1}{\pi}\ds\int_0^{\pi/\ve} \Psi(s) \,\cos(\ve\,s\,|\mk|)\,\md s.
$$

\paragraph{Two-dimensional case: $d=2$. } Let $r\geq0$ and $\mk\in \J_{n_x}$. In this case, {setting} $\mathbf{q} = -r\,\mk$, then using spherical coordinates, we have
\begin{align*}
{I}(\mk,r) & =\, \ds\int_{\mathbb{S}^1}\exp\left(
  i\,\mathbf{q}\cdot\omega \right)\md\omega    \, =\,\ds\int_0^{2\pi}
  \exp\left( i\,r\,\|\mk\|\,\cos \theta \right)\md\theta     \\ \,  &=\,2\ds\int_0^{\pi} \cos\left( r\,\|\mk\|\,\sin \theta \right)\md\theta   
\; =\,2\pi\, \mathcal{J}_0(r\,\|\mk\| ),
\end{align*}
where $\mathcal{J}_0$ is the Bessel function of order $0$, defined with
$$\mathcal{J}_0:x\in\R\mapsto \dfrac{1}{\pi}\ds\int_0^\pi \cos\left( x\,\sin\theta \right)\,\md\theta\,=\,\underset{l=0}{\overset{\infty}{\ds\sum}}\,\dfrac{(-1)^l}{(l!)^2}\,\left(\dfrac{x}{2}\right)^{2\,l}  .$$
Consequently, we get
\begin{align*}
\widehat\Psi_\ve(\mk) \,=\,\dfrac{1}{2\pi}\ds\int_0^{\pi/\ve} \Psi\left( s\right)\,s\,\mathcal{J}_0\left( \ve\,s\,\|\mk\| \right)\,\md s.
\end{align*}

\paragraph{Three-dimensional case: $d=3$. } Let $r\geq0$ and $\mk\in \J_{n_x}$. Hence, {setting} $\mathbf{q} = -r\,\mk$, and then using spherical coordinates, we get
\begin{align*}
{I}(\mk,r) & =\,\ds\int_{\mathbb{S}^2}\exp\left(
  i\,\mathbf{q}\cdot\omega \right)\md\omega   \,=\,2\pi\,\ds\int_0^\pi
  \exp\left( i\,\|\mathbf{q}\|\,\cos(\theta)
  \right)\,\sin\theta\,\md\theta    \\  &=\,2\pi\,\ds\int_{-1}^1 \exp\left( i\,\|\mathbf{q}\|\,\mu \right)\,\md\mu   
\; =\,4\pi\, \sinc\left( r\,\|\mk\| \right),
\end{align*}
where $\sinc(x)\,:=\,\sin(x)/x$. Thus, the kernel mode $\widehat\Psi_\ve(\mk)$ is given by
\begin{align*}
\widehat\Psi_\ve(\mk) \,=\,\dfrac{1}{2\pi^2}\ds\int_0^{\pi/\ve} \Psi\left( s \right)\,|s|^2\,\sinc\left( \ve\,s\,\|\mk\| \right)\,\md s.
\end{align*}

Now let us investigate the asymptotic behavior of the discrete operator $\mL_{n_x}^S[u]$ when $\ve\ll 1$.  To this {approach} we set $\mS_{n_x}$ the space of trigonometric polynomial of degree $n_x/2$ in each direction, defined as \cite{CHQZ}
$$
\mS_{n_x} \,=\,{\rm span}\left\{ \,e^{i\mk\cdot\mx}\,, \quad  -n_x/2 \leq \mk_j \leq n_x/2-1, \,\, j\,=\,1,\ldots,d\,\right\},  
$$
equipped with the classical $L^2$ norm $\|.\|_{L^2}$, which satisfies 
for any $u\in \mS_{n_x}$  
$$
\|u\|_{L^2}^2 \,=\, \left(\frac{2\pi}{n_x}\right)^d\, \sum_{j\in\J_{n_x}} |u(\mx_j)|^2
$$
and for any $u$ and $v\in \mS_{n_x}$, we also have
$$
\int_{\T} u(\mx)\,\overline{v}(\mx)\,\md\mx \,=\, \left(\frac{2\pi}{n_x}\right)^d\, \sum_{j\in\J_{n_x}} u(\mx_j)
\,\overline v(\mx_j). 
$$
Finally we define by $\mI_{n_x}$ the projection operator from
$\mC(\T)$ to $\mS_{n_x}$ such that $\mI_{n_x} u (\mx_j) = u(\mx_j)$,
for all $j\in\J_{n_x}$.

\begin{prop}
\label{prop:psi}
Let $d\in\{1,2,3\}$ and consider a connectivity
kernel $\Psi$ satisfying \eqref{hyp:psi} with
\begin{equation}
  \label{hyp:psi4}
\int_{\R^d} \Psi(\|\my\|)\,\|\my\|^4 \;\md\my \,<\, \infty.
\end{equation}
Then, for all $\mk\in \J_{n_x}$, there exists a positive constant $C>0$, depending on $\Psi$, such that for all $\ve>0$,
\begin{equation}
  \label{eq:limpsi}
\left|(2\pi)^d\,\widehat\Psi_\ve(\mk) \,-\,\overline{\Psi} \,+\,\overline{\sigma}\,\ve^2\,\|\mk\|^2   \right| \,\leq\, C\,\left(\|\mk\|^4+1\right)\,\ve^4\,.
\end{equation}
Moreover for any trigonometric polynomial $u \in \mS_{n_x}$, we have 
\begin{equation}
\label{eq:limLve}
\left\| \mL^S_{n_x}[u]  \,-\,\overline{\Psi}\,u
    -\overline{\sigma}\,\ve^2\,\Delta u \right\|_{L^2} \,\leq\, C
\,\ve^4\,\left(\|\Delta^2 u\|_{L^2} + \|u\|_{L^2}  \right).
\end{equation}
\end{prop}
\begin{proof}
 On the one hand, for any $\mk\in \J_{n_x}$, we perform a Taylor expansion of
 ${I}(\mk,.)$ at $r=0$ and using the assumptions  \eqref{hyp:psi4}  on $\Psi$, it yields
 \begin{eqnarray*}
&&\left|(2\pi)^d\,\widehat\Psi_\ve(\mk) \,-\,\int_0^{\pi/\ve }\Psi(s)
  s^{d-1}\md s
  \,-\,\ve^2\,  \|\mk\|^2\,\int_0^{\pi/\ve }\Psi(s)
   s^{d+1}\md s   \right| \\
   &&\leq\,
\|\mk\|^4\,\ve^4\,\int_{\R^d}\|\my\|^4\,\Psi(\|\my\|)\md \my.
\end{eqnarray*}
On the other hand, we have
$$
\int_{\pi/\ve}^\infty \Psi(s)
  s^{d-1}\md s \,+\,    \ve^2\int_{\pi/\ve}^\infty \Psi(s)
  s^{d+1}\md s \leq \ve^4 \left( \frac{1}{\pi^4} +\frac{1}{\pi^2} \right) \, \int_{\R^d}\|\my\|^4\,\Psi(\|\my\|)\md \my. 
  $$
  Gathering these results and using \eqref{hyp:psi}, there exists a constant $C>0$, depending on
  $\Psi$, such that
  $$
\left|(2\pi)^d\,\widehat\Psi_\ve(\mk) \,-\,\overline{\Psi} \,+\,\overline{\sigma}\,\ve^2\,\|\mk\|^2   \right| \,\leq\, C\,(\|\mk\|^4+1)\,\ve^4\,.
  $$
 Then, we consider  $u \in \mS_{n_x}$ and for $\mk\in\J_{n_x}$, we substitute the latter
 result in the expression \eqref{def:hatL0} of $\widehat{\mL}^S[u](\mk)$, it yields for each 
 $$
\left|\widehat{\mL}^S[u](\mk) \,-\, \left(\overline{\Psi}
  +\overline{\sigma}\,\ve^2\,\|\mk\|^2\right) \widehat{u}(\mk) \right|
\,\leq \, C\,\ve^2 \, (\|\mk\|^4+1)\,|\widehat{u}(\mk)|.
 $$
 Thus, from the definition of \eqref{def:L0}, we know that
 $\mL_{n_x}^S[u]\in \mS_{n_x}$ and get
 \begin{eqnarray*}
\|\mL_{n_x}^S[u] - \overline{\Psi} \,u - \overline{\sigma}\,\ve^2 \Delta u
   \|_{L^2} &= & \left(\sum_{\mk\in\J_{n_x}}  \left|\widehat{\mL}^S[u](\mk) \,-\, \left(\overline{\Psi}
  +\overline{\sigma}\,\ve^2\,\|\mk\|^2\right) \widehat{u}(\mk) \right|^2\right)^{1/2}\,, 
\\
   &\leq& C\,\ve^4\, (\|\Delta^2 u\|_{L^2} + \|u\|_{L^2})\,.
 \end{eqnarray*}
\end{proof}

\subsection{Particle/Spectral methods for \eqref{eq:kin-1}}
We now consider the transport equation \eqref{eq:kin-1} and apply a standard particle method. This kind of numerical scheme was first introduced
by Harlow \cite{HAR} for the numerical computation of specific
problems in fluid dynamics, and precisely mathematically studied later
\cite{RAV}. Thus a large diversity of particle methods were developed
for the simulation in fluid mechanics and plasma physics (see for
instance \cite{FIL}, \cite{FIL2} and references therein). The method consists in
approximating the solution $f$ to \eqref{eq:kin-1} with a sum of
Dirac masses centered in a finite number of solutions of the
characteristic system \eqref{eq:char}. These solutions stand for some
particles characterized by a pair membrane potential-adaptation
variable $(v,w)\in\R^2$. 

{In our case, since the transport equation
  \eqref{eq:kin-1} involves  a high dimensional space $(t,\mx,v,w)\in
  \R^+\times \R^d\times \R^2$, a particle method seems to be the most
  natural approach. Moreover, although the cost of the simulation
  increases with the number of particles considered, this kind of
  method has already shown its efficiency to describe complex dynamics in plasma physics and fluid dynamics.}

We approximate the solution $f$ to the transport
equation \eqref{eq:kin-1} at each point $\mx\in\T$,
$$
f_M(t,\mx,\md v,\md w) \,:=\,
\,\dfrac{\rho_0(\mx)}{M}\,\sum_{p=1}^{M}
\delta_{\V_{p}(t,\mx)}(\md
v)\otimes\delta_{\W_{p}(t,\mx)}(\md w),
$$
where  $M\in\mathbb{N}^*$, $\delta$ stands for the Dirac measure, and
for any $t\geq 0$, $(\V_{p},\W_{p})(t)\in \mS_{n_x}$ is the solution of the spatially discretized
characteristic system which can be written as follows,
$\mx\in\T$ and $1\leq p \leq M$
\begin{equation}
 \label{eq:chardiscret}
\left\{\begin{array}{l}
\dfrac{\md \V_{p}}{\md t} \,=\, \mI_{n_x}\left(N(\V_{p}) + \K_{n_x}^S[f_M](\V_{p})\right) \,-\, \W_{p}, \\ \, \\
\dfrac{\md \W_{p}}{\md t}\,=\, A\left( \V_p,\W_p\right),
\end{array}\right.
\end{equation}
with a given initial data $(\V^{0}_{p},\W^{0}_{p})\in \mS_{n_x}$ for
$1\leq p\leq M$ and $\mI_{n_x}$ is the projection operator on $\mS_{n_x}$.  Moreover,  we define the
macroscopic potential $V_M$ at each point $(t,\mx)\in\R^+\times\T$, as
\begin{equation}
  \left\{
    \begin{array}{l}
      \ds \rho_0 = \ds\int_{\R^2}f_M(t,\mx,\md v,\md w)\,,
      \\ \, \\
\ds\rho_0\,V_M(t,\mx) \,:=\, \ds\int_{\R^2}
v\,f_M(t,\mx,\md v,\md w) \,=\, \dfrac{1}{M}\,\sum_{p=1}^{M} \rho_0(\mx)\,\V_{p}(t,\mx)\,,
    \\ \, \\
   \ds\rho_0\,W_M(t,\mx) \,:=\, \ds\int_{\R^2}
w\,f_M(t,\mx,\md v,\md w) \,=\, \dfrac{1}{M}\,\sum_{p=1}^{M} \rho_0(\mx)\,\W_{p}(t,\mx)\,.
  \end{array}\right.
\label{my:macro}
\end{equation}
From these macroscopic quantities, it is then possible to compute the
discrete operator $\K_{n_x}^S[f_M]$ given in \eqref{def:Ksn}, where
\eqref{eq:chardiscret}--\eqref{my:macro} are solved at each mesh point $(\mx_\mj)_{\mj\in
  \J_{n_x}}$.

\subsection{Time discretization}
 The time discretization of \eqref{eq:chardiscret} is the key point to
 get an asymptotic preserving scheme.
 The basic idea is do develop numerical methods that
preserve the asymptotic limits ($\varepsilon\rightarrow 0$) from the microscopic to the macroscopic
models in the discrete setting. Contrary to multi-physics domain
decomposition methods, the asymptotic preserving schemes only solve the
microscopic equations avoiding the coupling of different models. This
approach  generates automatically  macroscopic
solvers when, in the asymptotic regime, the small time and space
scales are not resolved numerically. This idea can be illustrated in
Figure \ref{fig:ap}.

\tikzstyle{block} = [rectangle, draw, fill=blue!20, text width=6em, text centered, rounded corners, minimum height=3.5em]
\tikzstyle{line} = [line width=1.pt, draw, -latex']

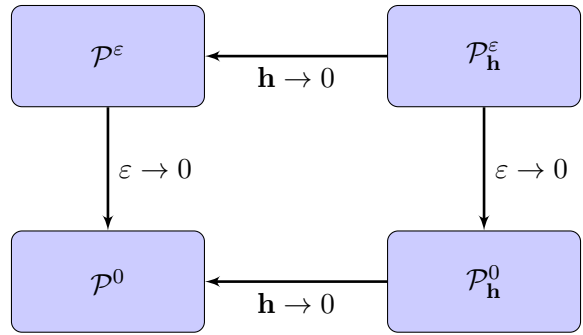
\begin{figure}[ht!]
 \begin{center}
  \begin{tikzpicture}[node distance = 2cm, auto]
 \node [block] (p1) {$\mathcal{P}^\ve$ };
 \node [block, right of=p1,  node distance=5cm] (p2) {$\mathcal{P}^\ve_\mh$};
 \node [block, below of=p1,  node distance=3cm] (p3) {$\mathcal{P}^0$};
\node [block, below of=p2,  node distance=3cm] (p4) {$\mathcal{P}^0_\mh$};
    \path [line]  (p2) -- node [] {$\mh \rightarrow 0$} (p1);
    \path [line] (p1) -- node [] {$\ve \rightarrow 0$} (p3);
    \path [line] (p4) -- node [] {$\mh \rightarrow 0$} (p3);
    \path [line] (p2) -- node [] {$\ve \rightarrow 0$} (p4);
  \end{tikzpicture}
  \end{center}
\caption{{\bf Asymptotic preserving diagram} is performed to evaluate
  uniform  error estimates $\|\mathcal{P}^\ve- \mathcal{P}_\mh^\ve\|$
  with respect to $\ve$. }
 \label{fig:ap}
\end{figure}

Suppose, we start with a microscopic model
$\mathcal{P}^\varepsilon$, which depends on a parameter $\varepsilon$,
characterizing the small scale. As $\varepsilon \rightarrow 0$, the
model is approximated by a macroscopic model
$\mathcal{P}^0$, which is independent of $\varepsilon$. We want to
design a discretization  $\mathcal{P}^\varepsilon_{\mh}$ of
$\mathcal{P}^\varepsilon$, where $\mh$ is the numerical parameter
(mesh size and time step). If the asymptotic limit
$\mathcal{P}^\varepsilon_{\mh}$, as $\varepsilon \rightarrow 0$ (with
a fixed $\mh$),  exists, then it is denoted by
$\mathcal{P}^0_{\mh}$. Furthermore when $\mathcal{P}^0_{\mh}$ is a
stable and consistent approximation of $\mathcal{P}^0$, then the
scheme $\mathcal{P}^\varepsilon_{\mh}$ is called asymptotic preserving.
Error on an asymptotic preserving scheme is obtained from the
following argument. Typically, we consider  $\mathcal{P}^\varepsilon$
and $\mathcal{P}^0$,  corresponding for instance to
\eqref{eq:kin-1}-\eqref{eq:kin-2} and its asymptotic model \eqref{eq:macro},
we expect formally \cite{CRE},
\begin{equation}
  \label{cont:anal}
\mathcal{E}_0({\ve}) \,:=\, \|  \mathcal{P}^\varepsilon - \mathcal{P}^0 \|
= O(\ve^2).
\end{equation}
Then we assume that $\mathcal{P}^\varepsilon_{\mh}$ is an $r$-order approximation of $\mathcal{P}^\varepsilon$ for a fixed $\ve>0$. Due to
the presence of the small parameter $\ve>0$, a classical numerical
analysis typically gives the following error estimates
\begin{equation}
  \label{class:anal}
\mathcal{E}_\mh(\ve) \,:=\,   \|  \mathcal{P}^\varepsilon_\mh - \mathcal{P}^\ve \|
= O(\|\mh\|^r/\ve^s), \quad s> 0,
\end{equation}
which blows-up when $\ve \ll 1$. Hence, the main issue of the asymptotic preserving analysis is to
establish the discrete counterpart of  the asymptotic
analysis \eqref{cont:anal}, that is, for a fixed $\mh$,
\begin{equation}
  \label{disc:anal}
\widetilde{\mathcal{E}}_\mh({\ve}) \,:=\, \|  \mathcal{P}^\varepsilon_\mh - \mathcal{P}^0_\mh \|
= O(\ve^2)
\end{equation}
and in the limit $\ve\rightarrow 0$,
\begin{equation}
  \label{disc0:anal}
  \mathcal{E}_\mh(0) \,:=\,  \|  \mathcal{P}^0_\mh - \mathcal{P}^0 \|
= O(\|\mh\|^r).
\end{equation}
Clearly, if we add up the error estimates \eqref{cont:anal},  \eqref{disc:anal}
and  \eqref{disc0:anal}, by the triangle inequality, we have
\begin{equation}
  \label{ap:anal}
  \mathcal{E}_\mh({\ve}) \,\leq \, \mathcal{E}_0({\ve}) +
  \widetilde{\mathcal{E}}_\mh({\ve}) + \mathcal{E}_\mh(0)
= O(\ve^2+\|\mh\|^r).
\end{equation}
By comparing the two error estimates \eqref{class:anal} and
\eqref{ap:anal}, it yields
$$
\mathcal{E}_\mh({\ve})  \leq C\, \min\left(\frac{\|\mh\|^r}{\ve^s},
  \,\ve^2 + \|\mh\|^r \right),
$$
showing that when $\ve \ll 1$, the error does not blow-up. This formal
argument applies to any asymptotic preserving schemes, although a
rigorous proof will be problem dependent, based on the regularity of
the solution to $\mathcal{P}^\ve$ and the specific scheme   $\mathcal{P}_\mh^\ve$.

Now, the aim is to apply this strategy  to  \eqref{eq:chardiscret}, which
corresponds to the characteristic curves of
\eqref{eq:kin-1}-\eqref{eq:kin-2}.  We have to be especially careful
about the stiff nonlocal terms in \eqref{eq:char}, where the small
parameter $\ve>0$ appears. We
cannot use a fully explicit scheme, which does not provide an
AP-scheme unless $\Delta t=O(\ve^2)$, whereas a fully implicit time discretization would be too costly because of the spectral collocation method for the nonlocal terms. 

Therefore, our strategy consists in applying implicit-explicit
numerical scheme, and to treat $\V_M$  as an additional unknown of the
system. In the following, we consider  $\Delta t>0$ and for all
$n\in\mathbb{N}$, we set $t^n\,=\, n \,\Delta t$.

In this section, we propose a first and a second order time
discretization scheme and prove some consistency properties when $\mh$ is fixed and $\ve
\rightarrow 0$ (see Proposition \ref{prop:cvRK1} and \ref{prop:cvRK2}
) corresponding to the error estimate $\widetilde{\mathcal{E}}_\mh(\ve)$ in \eqref{disc:anal} and when $\ve$ is fixed and $\mh
\rightarrow 0$ (see Lemma \ref{lem:cons1} and \ref{lem:cons2})
corresponding to the error estimates   ${\mathcal{E}}_\mh(\ve)$ in
\eqref{class:anal}.

\subsubsection*{A first order semi-implicit scheme}

We propose a first order semi-implicit scheme, that is for any time
step $n\in\mathbb{N}$ and any particle index $1\leq p\leq M$, we
approximate  $\left(\V_{p}(t^n),\W_{p}(t^n)\right)$ solution {of}
\eqref{eq:chardiscret} by $(\V_{p}^{n},\W^{n}_{p})\in \mS_{n_x}\times\mS_{n_x}$ given by the
following system 
\begin{equation}
\label{eq:RK1}
\left\{
  \begin{array}{l}
\ds\dfrac{\V_{p}^{n+1}-\V_{p}^{n}}{\Delta t} -
         \mI_{n_x}\left(N(\V_{p}^{n})+
    \frac{1}{\ve^2}\left[\mL^S_{n_x}[\rho_0V_M^n]  -
    \V_p^{n+1}\mL^S_{n_x}[\rho_0] \right]\right) + \W_{p}^{n} = 0\,,
         \\ \, \\
\dfrac{\W_{p}^{n+1}-\W_{p}^{n}}{\Delta t} - A\left(
    \V_{p}^{n+1},\W_{p}^{n} \right) = 0\,,
\end{array}\right.
\end{equation}
where $V^{n}_M$ denotes an approximation of the macroscopic membrane
potential. Using the linearity of $A$ and the fact that $(\V_{p}^{n},\W^{n}_{p})\in \mS_{n_x}\times\mS_{n_x}$, the system \eqref{eq:RK1} yields that $(\V_{p}^{n+1},\W^{n+1}_{p})\in \mS_{n_x}\times\mS_{n_x}$. Moreover, since the projection $\mI_{n_x}$ is linear, and $\mL^S_{n_x}[\rho_0\,V^n_M]\in \mS_{n_x}$ according to its definition \eqref{def:L0}, we get that the right term in the first equation in \eqref{eq:RK1} reads
\begin{multline*}
\mI_{n_x}\left(N(\V_{p}^{n})\,+\, \frac{1}{\ve^2}\,\left[\mL^S_{n_x}[\rho_0V_M^n]  \,-\, \V_p^{n+1}\,\mL^S_{n_x}[\rho_0] \right]\right)  \,-\, \W_{p}^{n} \\
=\,\mI_{n_x}\left(N(\V_{p}^{n})\right)\,+\, \frac{1}{\ve^2}\,\left[\mL^S_{n_x}[\rho_0V_M^n]  \,-\, \mI_{n_x}\left(\V_p^{n+1}\,\mL^S_{n_x}[\rho_0]\right) \right]  \,-\, \W_{p}^{n} .
\end{multline*}

On the one hand, let us emphasize that the stiff term, for $\ve\ll 1$, is treated implicitly but  can be solved exactly whereas other terms, nonlinear with respect to $\V_p$, are considered explicitly. Formally speaking, when $\ve$ tends to zero, at each point $\mx_j\in \T$, $j\in\J_{n_x}$, the microscopic potential $\V_p^{n+1}$ converges to $\mL^S_{n_x}[\rho_0V_M^n]/\mL^S_{n_x}[\rho_0]$.

On the other hand, the macroscopic membrane potential $V_M^n$ might be given by \eqref{my:macro} from the values $(\V_p^n)_{1\leq p\leq M}$. Unfortunately, this approach would not give the correct asymptotic behavior of the macroscopic membrane potential when $\ve \rightarrow 0$. {Indeed, as $\ve$ goes to $0$, the first equation in \eqref{eq:RK1} formally gives for all $p$ and $n$: 
  $$\V_p^{n+1}\,\sim\,V_M^n,$$
  which is the expected limit, but the non linear term $N(\V_p^n)$
  does not converge to $N(V_M^n)$. 
}

Therefore, we consider $V_M^n\in \mS_{n_x}$ as an additional variable solution {of} the following scheme
\begin{eqnarray}
  \label{eq:RK1J}
\dfrac{V^{n+1}_M-V^{n}_M}{\Delta t} &-& \ds
  \dfrac{1}{M}\underset{p=1}{\overset{M}{\sum}}\,\mI_{n_x}\left(
  N(\V_{p}^{n+1})\right) \\
  &-&  \frac{1}{\ve^2}\,\left[\mL^S_{n_x}[\rho_0V_M^n]  \,-\, \mI_{n_x}\left(V_M^{n}\,\mL^S_{n_x}[\rho_0]\right) \right]\,+\, W_{M}^{n} \,=\,0\,.
\nonumber
\end{eqnarray}
Observe here that the nonlinear term is computed implicitly from $(\V_p^{n+1})_{1\leq p\leq M}$ whereas the stiff term is now explicit.

Now, we define a numerical parameter $\mh\in\R^3$ as $\mh=(\Delta t, \Delta x, 1/M)$, where $\Delta x = 2\pi/n_x$ and let us show the consistency of the numerical scheme \eqref{eq:RK1}--\eqref{eq:RK1J} in the limit {as} $\ve\rightarrow 0$ for a fixed numerical parameter  $\mh$.
\begin{prop}[Consistency when $\ve\rightarrow 0$] 
  \label{prop:cvRK1}
  Let $\mh$ be a fixed parameter and consider a connectivity kernel  $\Psi:\R^+\rightarrow\R^+$ satisfying \eqref{hyp:psi}, \eqref{hyp:psi4} and a neuron density $\rho_0\in \mS_{n_x}$ satisfying \eqref{hyp:rho} at each grid point $\mx_j$, $j\in\J_{n_x}$. For all $\ve>0$, $p\in\{1,\ldots,M\}$ and $n\in\N$,  let us assume that the triplet $(\V_p^{\ve,n}, \,\W_p^{\ve,n}, V_M^{\ve, n})$ given by \eqref{eq:RK1}--\eqref{eq:RK1J} is uniformly bounded with respect to $\ve>0$. Then we define 
$$W_M^{\ve, n} = \frac{1}{M} \sum_{p=1}^M \W_p^{\ve,n}$$
and for all $j\in\J_{n_x}$, $(V_M^{\ve,n},W_M^{\ve,n})(\mx_j)$
converges to $(\overline V^n_M,\,\overline W^n_M)(\mx_j)$, as $\ve$
goes to $0$, solution {of}
\begin{equation}
  \label{eq:RK1lim}
\left\{\begin{array}{l}
\dfrac{\overline V_M^{n+1}-\overline V_M^{n}}{\Delta t} \,=\, \ds
         \mI_{n_x}\left( N\left(\overline V_M^{n}\right)\right)
         \,-\, \overline W_M^n   \\ \, \\
         \qquad\qquad\qquad\,+\, \overline{\sigma}  \,\left( \,\Delta
        \mI_{n_x}\left(\rho_0\,\overline V^n_M\right)
         - \mI_{n_x}\left(\overline V_M^n \Delta\rho_0  \right)\,\right),
         \\ \, \\
\dfrac{\overline W_M^{n+1}-\overline W_M^{n}}{\Delta t} \,=\, A\left(
         \overline V_M^n,\overline W_M^n \right).
\end{array}\right.
\end{equation}
\end{prop}
\begin{proof}
  For any $p\in\{1,\ldots,M\}$ and $n\geq 0$, we
 denote by  $\left(\V_p^{\ve,n},\,\W_p^{\ve,n},\,
   V_M^{\ve,n} \right)_{\ve>0}$ the solution {of}
 \eqref{eq:RK1}--\eqref{eq:RK1J} computed at the grid points $(\mx_j)_{j\in\J_{n_x}}$. Since this sequence, abusively labeled
 by $\ve$, is uniformly bounded, there exists a sub-sequence, still labeled in the same manner, which
 converges to  $\left(\overline\V_p^{n},\,
   \overline\W_p^{n},\, \overline V_M^{n}\right)$ when
 $\ve\rightarrow 0$.

 On the one hand using the scheme \eqref{eq:RK1} on $\V_p^{n+1}$, we may write
\begin{eqnarray*}
\ve^2\dfrac{\V_{p}^{\ve,n+1}-\V_{p}^{\ve,n}}{\Delta t} &=& \ve^2\,
                                                           \mI_{n_x}\left(
                                                           N(\V_{p}^{\ve,n})\right)\,-\,
                                                           \ve^2\,\W_{p}^{\ve,n}
   \\ &&+\,  \left[\mL^S_{\ve,n_x}[\rho_0V_M^{\ve,n}]  \,-\, \mI_{n_x}\left(\V_p^{\ve,n+1}\,\mL^S_{\ve,n_x}[\rho_0] \right)\right]\,,
\end{eqnarray*}
 and pass to the limit with respect to $\ve$, it yields that for any $j\in\J_{n_x}$,
 \begin{align*}
 &\mL^S_{\ve,n_x}[\rho_0V_M^{\ve,n}](\mx_j) \,-\,
   \mI_{n_x}\left(\V_p^{\ve,n+1}\,\mL^S_{\ve,n_x}[\rho_0]
   \right)(\mx_j) \\
   &=\, \mL^S_{\ve,n_x}[\rho_0V_M^{\ve,n}](\mx_j)  \,-\,
      \V_p^{\ve,n+1}(\mx_i)\,\mL^S_{\ve,n_x}[\rho_0](\mx_j)  \underset{\ve\rightarrow 0}\longrightarrow\, 0.
 \end{align*}
  Then, applying Proposition \ref{prop:psi} to $\rho_0\in \mS_{n_x}$,
 we have $\| \mL^S_{\ve,n_x}[\rho_0]- \overline\Psi\,\rho_0\|_{L^2}\rightarrow 0$, when $\ve$ goes to $0$, that is,   for
   any $j\in\J_{n_x}$
 $$
 \left| \mL^S_{\ve,n_x}[\rho_0](\mx_j) - \overline\Psi\,\rho_0(\mx_j) \right| \underset{\ve\rightarrow 0}\longrightarrow 0\,. 
 $$
 Furthermore,  applying again Proposition \ref{prop:psi} to
 $\mI_{n_x}(\rho_0V_M^{\ve,n})\in \mS_{n_x}$, we also get 
 $$
 \left|\mL^S_{\ve,n_x}[\rho_0 V_M^{\ve,n}](\mx_j) - \overline\Psi\,\rho_0
 \overline V_M^{n}(\mx_j) \right| \underset{\ve\rightarrow 0}\longrightarrow0\,, 
$$
hence for any $j\in\J_{n_x}$ and $p\in\{1,\ldots,M\}$, the limit $\overline\V_p^{n+1}(\mx_j)$ does no depend on $p$ and is given by
 $$
 \overline\V_p^{n+1}(\mx_j) \,=\,
 \left\{
   \begin{array}{ll}
     \overline V_M^{n}(\mx_j),  & {\rm if }\, \rho_0(\mx_j)>0,
     \\[0.9em]
     0, & {\rm else.}
 \end{array}\right.
 $$
 Now we consider $W_M^{\ve,n}$ given by
 $$
 W_M^{\ve,n} = \frac{1}{M} \sum_{p=1}^M \W_p^{\ve,n}
 $$
and apply the second relation in \eqref{eq:RK1}, it gives by linearity of $A$,
$$
\dfrac{W_{M}^{\ve,n+1}-W_{M}^{\ve,n}}{\Delta t} \,=\, A\left( \frac{1}{M}\sum_{p=1}^M\V_{p}^{\ve,n+1},W_{M}^{\ve,n}\right),
$$
Passing to the limit $\ve\rightarrow 0$, we get an equation on the
limit $\overline W_M^n$ given by
$$
\dfrac{\overline W_{M}^{n+1}-\overline W_{M}^{n}}{\Delta t} \,=\,
A\left( \overline V_{M}^{n},\overline W_{M}^{n} \right).
$$
On the other hand, we start from \eqref{eq:RK1J} and again apply
Proposition \ref{prop:psi}, it yields that
$$
\frac{V_M^{\ve,n} \,\left(\mL^S_{\ve,n_x}[\rho_0] - \overline\Psi \rho_0 \right)}{\ve^2}
\underset{\ve\rightarrow 0}\longrightarrow \overline{\sigma}\,\overline V_M^n \,\Delta\rho_0\,,
$$
whereas
$$
\frac{\mL^S_{\ve,n_x}[\rho_0 V_M^{\ve,n}] - \overline\Psi \,\mI_{n_x}\left(\rho_0 V_M^{\ve,n}\right)}{\ve^2}
\,\underset{\ve\rightarrow 0}\longrightarrow \overline{\sigma}\,\Delta\mI_{n_x}\left(\rho_0\overline V_M^n\right)\,.
$$
Gathering these latter results, we get that when $\ve$ goes to zero, 
$$
\frac{\mL^S_{\ve,n_x}[\rho_0 V_M^{\ve,n}]-  \mI_{n_x}\left(V_M^{\ve,n}
  \,\mL^S_{\ve,n_x}[\rho_0]\right)}{\ve^2} \underset{\ve\rightarrow 0}\longrightarrow 
\overline{\sigma}\,\left[ \Delta\mI_{n_x}\left(\rho_0\overline V_M^n\right)\,-\,  \mI_{n_x}\left(\overline V_M^n \Delta\rho_0\right) \right]. 
$$
Therefore the limit $\overline V_M^{n+1}$ is solution {of}
$$
\dfrac{\overline V_M^{n+1}-\overline V_M^{n}}{\Delta t} \,=\, \ds
\mI_{n_x}\left(N\left(\overline V_M^{n}\right)\right)
         \,-\, \overline W_M^n \,+\, \overline{\sigma}\left( \Delta
           \mI_{n_x}(\rho\,\overline V^n_M)
         - \mI_{n_x}\left(\Delta\rho_0 \, \overline V_M^n\right) \right).
       $$
       Finally, since the limit point  $(\overline V_M^n,\, \overline W_M^n)$  is uniquely determined, actually all the sequence $(V_M^{\ve,n},\, W_M^{\ve,n})_{\ve>0}$  converges.
\end{proof}
{Now, let us investigate the consistency error of the numerical scheme
\eqref{eq:RK1}--\eqref{eq:RK1J} as the parameter $\mh$ goes to $0$,
for a fixed parameter $\ve$ like $\mathcal{E}_\mh(\ve)$ in \eqref{class:anal}. Let us note $(\V^\ve,\W^\ve)$ the flow of the characteristic system \eqref{eq:char}, and $V^\ve$ the macroscopic potential as defined in \eqref{def:macro}. From the numerical scheme \eqref{eq:RK1}--\eqref{eq:RK1J}, we write the system in the form for all $n\in\mathbb{N}$:
$$
\F_1\left(\V_p^{\ve,n},\V_p^{\ve,n+1},\W_p^{\ve,n},\W_p^{\ve,n+1},V_M^{\ve,n},V_M^{\ve,n+1}
\right) \,=\, 0,
$$
where $\F_1:\R^6\rightarrow \R^3$. We define the consistency error
$\mathcal{E}_{\mh}^n(\ve)$ as
\begin{equation}
  \label{def:F1}
\mathcal{E}_{\mh}^n(\ve) \,:=\, \left\| F_1\left(\V_p^{\ve,n},\V_p^{\ve,n+1},\W_p^{\ve,n},\W_p^{\ve,n+1},V_M^{\ve,n},V_M^{\ve,n+1}\right)\right\|_{L^\infty},
 \end{equation} 
  where $\|.\|_{L^\infty}$ is the classical $L^\infty$ norm.

\begin{lem}[Consistency in the limit $\mh\rightarrow0$]
  \label{lem:cons1}
Let $0<\ve<1$ be a fixed parameter and consider $\Psi$ satisfying \eqref{hyp:psi}, and  $\rho_0$ satisfying \eqref{hyp:rho} and such that
$\text{Supp}(\rho_0)\subset \mathcal{B}(0,S)$ where $S>0$.  We suppose
that $f^\ve$ the solution of \eqref{eq:kin-1}--\eqref{eq:kin-2} is
differentiable twice with respect to time and there exists a
constant $C_T>0$, independent of $\ve$, such that in $[0,T]\times \mathcal{B}(0,S)$,
$$
\left\|\int_{\R^2} \left( v^4 + w^4\right)\, f^\ve(., \md v,\md w)
\right\|_{L^\infty} \,+\, \| V^\ve \|_{L^\infty} \,\leq \, C_T\,.
$$
Consider the scheme  \eqref{eq:RK1}--\eqref{eq:RK1J} and the consistency error
$\mathcal{E}_{\mh}^n(\ve)$ in \eqref{def:F1}. Then, there exists another positive constant $C>0$, independent of
$\mh$ and $\ve$, such
that for all $n\in \{0,\ldots\,, [T/\Delta t]\}$, 
$$
\mathcal{E}_{\mh}^n(\ve)\,\leq \, C\,\left( 1 + \frac{1}{\ve^2}\right)\left(
  \dfrac{\Delta\mx^{3/2}}{\ve^4}\,+\,\Delta t \,+\,
  \dfrac{1}{M}\right).
$$
\end{lem}
\begin{proof}
According to \eqref{eq:kin-1} and \eqref{eq:char}, we get that there exists a positive constant independent of $\ve$ such that for all $t\in[0,T]$ and $\mz=(\mx,v,w)\in\R^{d+2}$,
$$\left|\dfrac{\partial^2}{\partial t^2}\V^\ve (t,\mz)\right| \,+\, \left|\dfrac{\partial^2}{\partial t^2}\W^\ve (t,\mz)\right| \,+\, \left|\dfrac{\partial^2}{\partial t^2}(\rho_0\,V^\ve) (t,\mx)\right| \,\leq\, \dfrac{C}{\ve^4}.$$
Thus, from a Taylor expansion of the solution to \eqref{eq:char} and
the first equation of \eqref{eq:macroeps}, the consistency error can
be bounded  as
$$
\mathcal{E}_{\mh}^n(\ve) \,\leq\, \mathcal{T}_1 \,+\, \mathcal{T}_2 \,+\,
\mathcal{T}_3 \,+\, \mathcal{T}_4 \,+\,C\dfrac{\Delta t}{\ve^4}.
$$
with the non stiff terms
$$
\left\{\begin{array}{l}
\mathcal{T}_1 :=\left\|\,
         A\left(\V^{\ve}(t^{n+1}),\W^\ve(t^n)\right) \,-\,
         A\left(\V^{\ve}(t^{n}),\W^\ve(t^n)\right) \,\right\|\,,
         \\ \, \\
\mathcal{T}_2 :=\ds\left\|\, N(\V^\ve(t^n)) \,-\,
         \mathcal{I}_{n_x}\left(N(\V^\ve(t^n))\right)
         \,\right\|\,,
         \\ \, \\
\mathcal{T}_3 := \left\| \frac{\rho_0(\mx)}{M}\sum_{p=1}^M \mI_{n_x}\left( N(\V^{\ve}(t^{n+1}))\right) - \int_{\R^2} N(\V^\ve(t^n))\,f^\ve_0(\md v',\md w') \right\|\,,
       \end{array}\right.
     $$
and the stiff term
$$\begin{array}{ll}
\mathcal{T}_4  &:= \dfrac{1}{\ve^2}\ds\left\|\,  \mI_{n_x}\left(\mL^S_{n_x}[\rho_0]\,\V^{\ve}(t^{n+1})\right) \,-\, \Psi_\ve\star[\rho_0]\,\V^\ve(t^n) \,\right\|   \\[0.9em]
& +\,\dfrac{1}{\ve^2}\,\left\| \mI_{n_x}\left(\mL^S_{n_x}[\rho_0\,V^\ve(t^n)]\right) \,-\, \Psi_\ve\star[\rho_0\,V^\ve] \,\right\|    \\[0.9em] 
& + \,  \dfrac{1}{\ve^2}\,\left\|  \,\mI_{n_x}\left(\mL^S_{n_x}[\rho_0]\,V^{\ve}(t^n)\right) \,-\, \Psi_\ve\star[\rho_0]\,V^\ve(t^n) \,\right\|\,.
\end{array}$$
First of all, for the first term, since second and fourth order moments of $f^\ve$  in $v$ and $w$ are uniformly bounded with respect to $\ve$, we directly get that
$$
\mathcal{T}_1 \,\leq\, C\,\Delta t,
$$
for some positive constant $C>0$. Then, in order to treat the second term, we use the estimate from Theorem 2.12 in \cite{HES} which yields that for all $U\in H^2(\B(0,S))$, there exists a constant $C>0$ such that 
$$\|U \,-\, \mathcal{I}_{n_x}(U)\|_{L^\infty(\B(0,S))}\,\leq\, C\,\|U\|_{H^2(\B(0,S))}\,\Delta\mx^{3/2}.$$
Since the initial data are regular enough, we get that there exists a positive constant independent of $\ve$ such that for all $t\in[0,T]$ and $\mz=(\mx,v,w)\in\B(0,S)\times\R^{2}$,
$$
\left|\nabla_\mx^2\V^\ve (t,\mz)\right| \,+\,
\left|\nabla_\mx^2\W^\ve (t,\mz)\right| \,+\, \left|\nabla_\mx^2
  (\rho_0\,V^\ve) (t,\mx)\right| \,\leq\, \dfrac{C}{\ve^4}.
$$
This leads to the estimate
$$
\mathcal{T}_2\,\leq\,C\,\dfrac{\Delta \mx^{3/2}}{\ve^4},
$$
for some positive constant $C>0$. As for the third term, we decompose it as follows:
\begin{align*}
\mathcal{T}_3 \,&\leq\,\ds\left\| \dfrac{\rho_0}{M}\sum_{p=1}^M|\mI_{n_x}\left( N(\V^{\ve}(t^{n+1})\right) - N(\V^{\ve}(t^{n+1}))| \right\|   \\
&+\,\left\| \dfrac{\rho_0}{M}\sum_{p=1}^M| N(\V^{\ve}(t^{n+1}) - N(\V^{\ve}(t^{n})| \right\| \\
&+\,\left\| \dfrac{\rho_0}{M}\sum_{p=1}^MN(\V^{\ve}(t^{n}) -\int_{\R^2} N(\V^\ve(t^n))\,f^\ve_0(\md v'\,\md w' )\right\|\,.
\end{align*}
The first term can be treated as previously with the estimate from
\cite{HES}. Then, the second term is of order $\Delta t$. Finally, the
third term corresponds to an approximation of the integral term with
the rectangle rule, hence it is of order $1/M$. Consequently, we get that
$$
\mathcal{T}_3\,\leq\,C\left( \dfrac{\Delta\mx^{3/2}}{\ve^4}\,+\,\Delta t
  \,+\, \dfrac{1}{M} \right),
$$
for some positive constant $C$. As for the final term $\mathcal{T}_4$, we notice that for all $U\in \mathcal{C}(\B(0,S))$, 
$$
\mL^S_{n_x}[U] \,=\, \mI_{n_x}\left[ \Psi_\ve\star U \right].
$$
Therefore, we want to use the estimate form \cite{HES} here again. This gives us that there exists a positive constant $C$ such that
$$
\mathcal{T}_4  \,\leq\, \dfrac{C}{\ve^2}\left( \dfrac{\Delta\mx^{3/2}}{\ve^4}\,+\,\Delta t \,+\, \dfrac{1}{M}\right).
$$
Consequently, gathering the previous results, we get that there exists a positive constant $C$ independent of $\mh$ and $\ve$ such that if $\|\mh\|$ is small enough,
$$
\mathcal{E}_{\mh}^n(\ve)\,\leq \, C\,\left( 1 + \frac{1}{\ve^2}\right)\left( \dfrac{\Delta\mx^{3/2}}{\ve^4}\,+\,\Delta t \,+\, \dfrac{1}{M}\right).
$$ 
\end{proof}
}

The lack of uniform bounds, with respect to $\ve$ and $\mh$, on the
numerical solution does not allow us to complete the rigorous analysis
of the asymptotic preserving scheme
\eqref{eq:RK1}--\eqref{eq:RK1J}. Furthermore, Proposition \ref{prop:cvRK1}
only gives a convergence result without any error estimate as $\widetilde{\mathcal{E}}_\mh(\ve)$ in
\eqref{disc:anal}. However,  Proposition  \ref{prop:cvRK1} indicates that in the limit $\ve\rightarrow 0$,  the
numerical scheme \eqref{eq:RK1}--\eqref{eq:RK1J} becomes a first
order explicit  time approximation with respect to $\Delta t$ of the reaction-diffusion system
\eqref{eq:macro}, hence it constitutes together with Lemma \ref{lem:cons1}
a first direction to provide a proof of the asymptotic preserving
property \eqref{ap:anal}.

\subsubsection*{A second order implicit-explicit Runge-Kutta scheme}

Now let us adapt the previous strategy to a second order implicit-explicit Runge-Kutta  scheme for the system \eqref{eq:chardiscret}. We propose a combination of Heun's method for the explicit part, and an A-stable second order singly diagonally implicit Runge-Kutta (SDIRK) method for the implicit part. According to the classification from \cite{BOS}, we call it H-SDIRK2 (2,2,2). 

For all $n\in\mathbb{N}$ and $p\in\{1\ldots M\}$, we apply a first stage,
\begin{equation}
  \label{eq:RK2-1}
  \left\{
    \begin{array}{ll}
                     \ds \V_{p}^{(1)} &= \V_{p}^{n} \,+\, \\ \,\\
      & \ds\frac{\Delta t}{2}\left[ \mI_{n_x}\left(N(\V_{p}^{n})\right)\,+\, \frac{1}{\ve^2}\,\left[\mL^S_{n_x}[\rho_0V_M^n]  \,-\, \mI_{n_x}\left(\V_p^{(1)}\,\mL^S_{n_x}[\rho_0]\right) \right] \,-\,
           \W_{p}^{n} \right]\,,
           \\\,\\
           \ds \W_{p}^{(1)} &= \W_{p}^{n} \,+\,  \frac{\Delta t}{2}\,A\left( \V_{p}^{(1)},\W_{p}^{n} \right),
 \end{array}\right.
\end{equation}
Hence we  compute the  additional variable ${V}_M^{(1)}\in \mS_{n_x}$  solution {of} the following scheme
\begin{eqnarray}
  \label{eq:RK2J-1}
  &&V^{(1)}_M \,=\, V^{n}_M \,+\,
  \\ &&\ds \frac{\Delta t}{2}\left[ \dfrac{1}{M}\underset{p=1}{\overset{M}{\sum}}\mI_{n_x}\left( N(\V_{p}^{(1)})\right) \,+\,  \frac{1}{\ve^2}\,\left[\mL^S_{n_x}[\rho_0V_M^n]  \,-\, \mI_{n_x}\left(V_M^{n}\,\mL^S_{n_x}[\rho_0]\right) \right]\,-\, W_{M}^{n}\right]\,,
\nonumber
\end{eqnarray}
with
$$
W_M^{n} := \frac{1}{M} \;\sum_{p=1}^M \W_p^{n}. 
$$
Then, we set
$$
\left\{\begin{array}{l}
       \hat\V_p^{(1)} \,\,=\,\, 2\;\V_{p}^{(1)}\,-\,\V_{p}^{n}, \\[0.9em]
       \hat\W_p^{(1)} \,\,=\,\, 2\,\W_{p}^{(1)}\,-\,\W_{p}^{n},\\[0.9em]
       \hat V_M^{(1)} \,\,=\,\, 2 \,V_{M}^{(1)}\;-\; V_{M}^{n}
       \end{array}\right.
$$
and compute the second stage with a semi-implicit step on $(\V_p^{(2)},\W_p^{(2)})$,
\begin{equation}
  \label{eq:RK2-2}
  \left\{
    \begin{array}{ll}
           \ds \V_{p}^{(2)} &= \V_{p}^{n} \,+\, \\ \, \\ & \ds\frac{\Delta t}{2}\left[ \mI_{n_x}\left(N(\hat\V_p^{(1)})\right)\,+\, \frac{1}{\ve^2}\,\left[\mL^S_{n_x}[\rho_0\hat V_M^{(1)}]  \,-\, \mI_{n_x}\left(\V_p^{(2)}\,\mL^S_{n_x}[\rho_0]\right) \right]\,-\,
           \hat\W_{p}^{(1)} \right]\,,
           \\\,\\
           \ds \W_{p}^{(2)} &= \W_{p}^{n} \,+\,  \frac{\Delta t}{2}\,A\left( \V_{p}^{(2)},\hat\W_{p}^{(1)} \right),
 \end{array}\right.
\end{equation}
Moreover, ${V}_M^{(2)}\in \mS_{n_x}$ is given by
\begin{eqnarray}
  \label{eq:RK2J-2}
  &&V^{(2)}_M = V^{n}_M \,+\,
  \\
  &&\ds \frac{\Delta t}{2}\left[ \dfrac{1}{M}\underset{p=1}{\overset{M}{\sum}}\mI_{n_x}\left( N(\V_{p}^{(2)})\right) \,+\,  \frac{1}{\ve^2}\,\left[\mL^S_{n_x}[\rho_0\hat V_M^{(1)}]  \,-\, \mI_{n_x}\left(\hat V_M^{(1)}\,\mL^S_{n_x}[\rho_0]\right) \right]\,-\, \hat W_{M}^{(1)}\right]\,,
\nonumber
\end{eqnarray}
where $\hat W_M^{(1)}= 2 W_M^{(1)}- W_M^n$ and
$$
W_M^{(1)} := \frac{1}{M} \;\sum_{p=1}^M \W_p^{(1)}. 
$$
Finally, we get the numerical solution at time $t^{n+1}$ through
\begin{equation}
  \label{eq:RK2-3}
\left\{\begin{array}{l}
\V^{n+1}_{p} \,=\, \V^{(1)}_{p} \,+\, \V^{(2)}_{p} \,-\, \V^{n}_{p}, \\ \, \\
\W^{n+1}_{p} \,=\, \W^{(1)}_{p} \,+\, \W^{(2)}_{p} \,-\, \W^{n}_{p}, \\ \, \\
V^{n+1}_{M} \,=\, V^{(1)}_{M} \,+\, {V}^{(2)}_{M} \,-\, V^{n}_{M}.
\end{array}\right.
\end{equation}
Now, we prove an analogous result to Proposition \ref{prop:cvRK1} for
the numerical scheme \eqref{eq:RK2-1}--\eqref{eq:RK2-3} as the
parameter $\ve$ goes to $0$ with a fixed  numerical parameter
$\mh\in\R^3$ given by $\mh=(\Delta t, \Delta x, 1/M)$, where
$\Delta x = 2\pi/n_x$, $\Delta t>0$  and  $M\in\N^*$.

\begin{prop}[Consistency when $\ve\rightarrow 0$] 
  \label{prop:cvRK2}
  Let $\mh$ to be  fixed and consider
a connectivity kernel  $\Psi:\R^+\rightarrow\R^+$ satisfying \eqref{hyp:psi},
\eqref{hyp:psi4} and a neuron density $\rho_0\in \mS_{n_x}$ satisfying
\eqref{hyp:rho} at each grid point $\mx_j$, $j\in\J_{n_x}$. For all $\ve>0$, $p\in\{1,\ldots,M\}$ and $n\in\N$,  let us assume that the triplet $(\V_p^{\ve,n},
\,\W_p^{\ve,n}, V_M^{\ve, n})$   given by \eqref{eq:RK2-1}--\eqref{eq:RK2-3}
is uniformly bounded with respect to $\ve>0$. Then we define
$$
W_M^{\ve, n} = \frac{1}{M} \sum_{p=1}^M \W_p^{\ve,n}
$$
and for all $j\in\J_{n_x}$, $(V_M^{\ve,n},W_M^{\ve,n})(\mx_j)$ converges to
$(\overline V^n_M,\,\overline W^n_M)(\mx_j)$,  as $\ve$ goes to $0$,
solution {of}
\begin{equation}
  \label{eq:RK2lim-1}
\left\{\begin{array}{ll}
         \overline V^{(1)}_M & =\, \overline V^{n}_M \,+\\ \;\\ & \ds\dfrac{\Delta t}{2}\left[\mI_{n_x}\left( N\left(\overline V_M^{n}\right)\right)
         - \overline W_M^n + \overline{\sigma}  \left( \Delta
        \mI_{n_x}\left(\rho_0\,\overline V^n_M\right)
         - \mI_{n_x}\left(\overline V_M^n \Delta\rho_0  \right)\right)\right], \\ \, \\
\overline W^{(1)}_M &=\, \overline W^{n}_M \,+\, \dfrac{\Delta t}{2}\,A\left( \overline V^n_M , \overline W^n_M\right), 
\end{array}\right.
\end{equation}
where the second stage is given by
\begin{equation}
  \label{eq:RK2lim-2}
\left\{\begin{array}{ll}
              \overline  V^{(2)}_M &=\, \overline V^{n}_M \,+\\ \,\\ &\ds \dfrac{\Delta t}{2}\left[ N\left(\hat V^{(1)}_M \right) -\hat W^{(1)}_M + \overline{\sigma}  \left( \Delta
        \mI_{n_x}\left(\rho_0\hat V^{(1)}_M\right)
        - \mI_{n_x}\left(\hat V_M^{(1)} \Delta\rho_0  \right)\,\right)\right]
\\ \, \\
      W^{(2)}_M &=\, W^{n}_M \,+\, \dfrac{\Delta t}{2}\,A\left(\hat V^{(1)}_M , \hat W^{(1)}_M\right), 
\end{array}\right.
\end{equation}
where $\hat  V^{(1)}_M = 2  \overline V^{(1)}_M -  \overline V^{n}_M$,  $\hat  W^{(1)}_M = 2  \overline W^{(1)}_M -  \overline W^{n}_M$. The next time step is given by
\begin{equation}\label{eq:RK2lim-3}
\left\{\begin{array}{l}
V^{n+1}_{M} \,=\, V^{(1)}_{M} \,+\, V^{(2)}_{M} \,-\, V^{n}_M, \\ \, \\
W^{n+1}_{M} \,=\, W^{(1)}_{M} \,+\, W^{(2)}_{M} \,-\, W^{n}_{M}.
\end{array}\right.
\end{equation}
\end{prop}
\begin{proof}
We fix a time step $\Delta t >0$, a
set of equidistant points $(\mx_j)_{j\in\J_{n_x}}\subset \T$ and
$p\in\{1,\ldots,M\}$.   Then we denote by $\left(\V_p^{\ve,n},\,\W_p^{\ve,n},\,
  V_M^{\ve,n} \right)_{\ve>0}$ the solution {of}
\eqref{eq:RK2-1}--\eqref{eq:RK2-3}. Up to a sub-sequence,  $\left(\V_p^{\ve,n},\,\W_p^{\ve,n},\,
  V_M^{\ve,n} \right)_{\ve>0}$ converges to $\left(\overline\V_p^{n},\,
  \overline\W_p^{n},\, \overline V_M^{n}\right)$ when $\ve\rightarrow
0$, hence we proceed exactly as in Proposition \ref{prop:cvRK1}  and set
$$
W_M^{\ve,(1)} = \frac{1}{M} \sum_{p=1}^M \W_p^{\ve,(1)}  \underset{\ve\rightarrow 0}\longrightarrow \overline W_M^{(1)} .
$$
Thus, we prove that $\left(\overline V_M^{(1)}, \overline
  W_M^{(1)}\right)$ corresponds to the solution of the first stage
\eqref{eq:RK2lim-1} and  we have
$$
\left\{\begin{array}{l}
         \overline V_M^{(1)} \,\,=\,\, 2\, \overline V_M^{(1)} \,-\, \overline V_M^{n}, \\ \, \\
         \overline W_M^{(1)} \,\,=\,\, 2\,\overline W_M^{(1)} \,-\, \overline W_M^{n}.
         \end{array}\right.
$$
Furthermore,  we treat the second stage in the same manner for any $j\in\J_{n_x}$ and $p\in\{1,\ldots,M\}$, the limit $\overline\V_p^{(2)}(\mx_j)$ does not depend on $p$ and is given by
 $$
 \overline\V_p^{(2)}(\mx_j) \,=\,
 \left\{
   \begin{array}{ll}
     \hat V_M^{(1)}(\mx_j),  & {\rm if }\, \rho_0(\mx_j)>0,
     \\[0.9em]
     0, & {\rm else.}
 \end{array}\right.
$$
Passing to the limit {as} $\ve\rightarrow 0$ in \eqref{eq:RK2J-2} and in the second equation in \eqref{eq:RK2-2}, it yields that $(\overline V_M^{(2)}, \overline W_M^{(2)})$ satisfies \eqref{eq:RK2lim-2} and finally  \eqref{eq:RK2lim-3}.
\end{proof}

Let us notice that the present strategy can be applied to a large
class of second order schemes and can also be extended to a third
order semi-implicit scheme. We refer to \cite{BOS} for the detailed
description of the schemes.

{
Now, let us investigate the consistency with respect to the numerical parameter $\mh$. 
\begin{lem}[Consistency $\mh\rightarrow 0$]
  \label{lem:cons2}
Let $0<\ve<1$ be a fixed parameter and consider $\Psi$ satisfying \eqref{hyp:psi}, and  $\rho_0$ satisfying \eqref{hyp:rho} and such that
$\text{Supp}(\rho_0)\subset \mathcal{B}(0,S)$ where $S>0$.  We suppose
that $f^\ve$ the solution of \eqref{eq:kin-1}--\eqref{eq:kin-2} is
differentiable twice with respect to time and there exists a
constant $C_T>0$, independent of $\ve$, such that in $[0,T]\times \mathcal{B}(0,S)$,
$$
\left\|\int_{\R^2} \left( v^4 + w^4\right)\, f^\ve(., \md v,\md w)
\right\|_{L^\infty} \,+\, \| V^\ve \|_{L^\infty} \,\leq \, C_T\,.
$$
Consider the scheme  \eqref{eq:RK2-1}--\eqref{eq:RK2-3} and the consistency error
$\mathcal{E}_{\mh}^n$ in \eqref{def:F1}. Then, there exists another positive constant $C>0$, independent of
$\mh$ and $\ve$, such
that for all $n\in \{0,\ldots\,, [T/\Delta t]\}$, 
$$
\mathcal{E}_{\mh}^n(\ve)\,\leq \, C\,\left( 1 + \frac{1}{\ve^2}\right)\left(
  \dfrac{\Delta\mx^{3/2}}{\ve^4}\,+\,\Delta t^2 \,+\,
  \dfrac{1}{M}\right).
$$
\end{lem}
\begin{proof}
  The proof uses the same tools as in the proof of Lemma \ref{lem:cons1}.
\end{proof}

As in the previous section,  Proposition  \ref{prop:cvRK2} indicates that in the limit $\ve\rightarrow 0$,  the
numerical scheme \eqref{eq:RK2-1}--\eqref{eq:RK2-3} becomes a second
order explicit  time approximation with respect to $\Delta t$ of the reaction-diffusion system
\eqref{eq:macro}. Applying  Lemma \ref{lem:cons2}, we may conjecture
that the asymptotic preserving
property \eqref{ap:anal} is satisfied.

In the following section, we provide some numerical evidences on  this issue.

\section{Numerical simulations}
\label{sec:num}
\setcounter{equation}{0}

In this section, we provide examples of numerical computations to validate and compare the different time discretization schemes  \eqref{eq:RK1}--\eqref{eq:RK1J} and \eqref{eq:RK2-1}--\eqref{eq:RK2-3} introduced in the previous  section.

First of all, we focus on the order of accuracy when $\ve$ is fixed
and the numerical parameter $\mh$ goes to zero. Then we study the
behaviour of the numerical solutions for a fixed $\mh$ and in the limit $\ve\rightarrow 0$, to show the convergence towards the solutions of the approximations \eqref{eq:RK1lim} and \eqref{eq:RK2lim-1}--\eqref{eq:RK2lim-3} of the reaction-diffusion system \eqref{eq:macro}.

Then, we display some simulations of the behaviour of a solution {of} \eqref{eq:kin-1}--\eqref{eq:kin-2} with a heterogeneous neuron density, and finally, we show some two-dimensional dynamics.

Throughout this section, except for the first subsection, we fix the parameter of the nonlinearity $N$ to $\theta=0.1$ and the other constants to $\tau=0.005$ and $\gamma=5$, expect in the first subsection. This framework corresponds to the ``excitable'' regime of the well-known FHN reaction-diffusion system \eqref{eq:FHNreac}. Therefore,  the system only admits  one steady state which is the stable fixed point $0$, and according to \cite{CAR}, $\tau$ is small enough so that the solution {of} \eqref{eq:FHNreac} exhibits slow/fast dynamics like traveling pulses.

Moreover, as for the connectivity kernel, we use the following truncated gaussian function
\begin{equation}
  \label{gauss}
\Psi(\|\mz\|)\,=\;\dfrac{1}{(2\,\pi\,\sigma_0)^{d/2}}\,\exp\left( -\dfrac{\|\mz\|^2}{2\,\sigma_0} \right),
\end{equation}
with  $\sigma_0=0.005$ such that we have in \eqref{hyp:psi},
$$
\overline{\Psi} = 1 \quad{\rm and}\quad \overline{\sigma}= \frac{\sigma_0}{2}. 
$$

\subsection{Order of accuracy in the numerical parameters}

In this subsection, we aim to verify the order of accuracy  of our
numerical methods proposed in Section \ref{sec:particle} with respect
to the numerical parameters $\mh=(\Delta t,\Delta x,1/M)$, when it
goes to zero. We consider a simplified version of the nonlocal transport equation \eqref{eq:kin-1} with $N(v)=-\alpha v$ and $\tau=0$, that is, for $t>0$ and $\mx\in\R$
\begin{equation}\label{eq:fsimple}
\left\{\begin{array}{l}
\partial_t f^\ve \,+\,\partial_v\left( f^\ve\left[ -\alpha\,v \,-w\,+\, \K_\ve[f^\ve] \right] \right) \,=\,0, \\ \, \\
f^\ve|_{t=0}(\mx,v,w) \,=\,\delta_{V_0(\mx)}(v) \otimes \delta_{0}(w),
\end{array}\right.
\end{equation}
with $V_0$ given by
$$
V_0(\mx)\,=\,\exp\left( -100\,|\mx|^2  \right),\quad \mx \in\R.
$$
Consequently, in this configuration, we have $\rho_0\equiv 1$, and the solution {of} \eqref{eq:kin-1}--\eqref{eq:kin-2} is given by $f^\ve = \delta_{V^\ve}(v)\,\otimes \delta_0(w)$ where $V^\ve$ is the unique solution {of} the following reaction-diffusion equation for $t>0$ and $\mx\in\R$,
$$\left\{\begin{array}{l}
\partial_t V^\ve \,-\, \dfrac{1}{\ve^2}\left( \Psi_\ve\star V^\ve \,-\,
           \overline{\Psi}\,V^\ve \right) \,=\, -\alpha\,V^\ve,
           \\ \, \\
V^\ve(0,\mx)\,=\,V_0(\mx),\quad \mx\in\R.
         \end{array}\right.
       $$
{Thus, the parameter $\alpha>0$ determines the rate of convergence of $V^\ve$ towards the stable state $0$.}
Since the term $N(V^\ve)$ is now linear, the macroscopic equation on $V^\ve$ is also linear (even if the equation on $f^\ve$ is not) and we can exhibit an explicit solution using a Fourier transform in space. It yields that,
$$
\widehat{V}^\ve(t,\xi) \,=\, \widehat{V_0}(\xi)\,\exp\left( \left[
    -\alpha\,+\,\dfrac{1}{\ve^2}\left( \widehat{\Psi}_\ve(\xi) - \overline{\Psi}
    \right) \right]\,t \right).
$$
where we choose the parameter $\alpha = 0.001$, and the scaling parameter
$\ve=1$. The domain in space is taken to be $(-1,1)$. We compute an approximation of the error on the macroscopic quantity $V^\ve$ at each time step
$$
\mathcal{E}^n \,=\, \| V^{\ve,n}_M  - V^\ve(t^n)\|_{L^2}, \quad n=0,\ldots N_T,  
$$
with $N_T=[T/\Delta t]$. In Table \ref{tab:01} and \ref{tab:02}, we report the numerical error for different values of $\mh$ at fixed time $T=10$ for the numerical schemes \eqref{eq:RK1}--\eqref{eq:RK1J} ({first table}) and \eqref{eq:RK2-1}--\eqref{eq:RK2-3} ({second table}). A linear regression yields that these numerical methods seems to be respectively first and second order in $\mh$. Therefore, with this parametrization, the order of accuracy  corresponds to the one given by the time discretization, whereas the error due to the spectral discretization is negligible.

\begin{table}
  \begin{center}
\begin{tabular}{|c|c|c|}
  \hline\makebox[3em]{$\|\mh\|$}&\makebox[12em]{$L^2$ error for \eqref{eq:RK1}--\eqref{eq:RK1J} } &
                                                        \makebox[3em]{Order}
  \\\hline\hline
  1.e-01 & 5.48e-04 & XXX
  \\\hline
  5.e-02 & 2.73e-04 & 1.63
  \\\hline
  2.e-02 & 1.09e-04 & 1.00
  \\\hline
  1.e-02 & 5.47e-05 & 1.00
  \\\hline
  5.e-03 & 2.73e-05 & 1.00
  \\\hline
  2.e-03 & 1.09e-05 & 1.00
  \\\hline
  1.e-03 & 5.47e-06 & 1.00
  \\\hline
  5.e-04 & 2.73e-06 & 1.00
  \\\hline
  \end{tabular} 
\end{center}
\caption{{\bf Order of accuracy in $\|\mh\|\rightarrow 0$:} evaluation of the
  numerical  error at fixed time $T=10$ of the numerical schemes \eqref{eq:RK1}--\eqref{eq:RK1J}.}
\label{tab:01}
\end{table}

\begin{table}
  \begin{center}
\begin{tabular}{|c|c|c|}
  \hline\makebox[3em]{$\|\mh\|$}&\makebox[12em]{$L^2$ error  for \eqref{eq:RK2-1}--\eqref{eq:RK2-3}} &
                                                        \makebox[3em]{Order}
  \\\hline\hline
  1.e-01 & 1.23e-07 & XXX
  \\\hline
  5.e-02 & 3.56e-08 & 2.69
  \\\hline
  2.e-02 & 8.35e-09 & 2.02
  \\\hline
  1.e-02 & 2.07e-08 & 2.01
  \\\hline
  5.e-03 & 5.01e-09 & 2.01
  \\\hline
  2.e-03 & 1.23e-09 & 2.01
  \\\hline
  1.e-03 & 2.95e-10 & 2.01
  \\\hline
  5.e-04 & 2.95e-10 & 2.00
  \\\hline
 \end{tabular}
\end{center}
\caption{{\bf Order of accuracy in $\|\mh\|\rightarrow 0$:} evaluation of the
  numerical  error at fixed time $T=10$ of the numerical schemes \eqref{eq:RK2-1}--\eqref{eq:RK2-3}.}
\label{tab:02}
\end{table}

\subsection{Order of accuracy in $\ve$}
We again consider the transport equation
\eqref{eq:kin-1}--\eqref{eq:kin-2},  with the initial data 
\begin{equation}\label{eq:f0}
f_0(\mx,v,w) = \delta_{V_0(\mx)}(v)\,\otimes\,\delta_{W_0(\mx)}(w),
\end{equation}
with
$$
V_0 \,=\, \chi_{[-1,1]} \quad{\rm and} \quad W_0\,\equiv\,0.$$
In this configuration, we get $\rho_0\equiv 1$ and the solution of the
transport equation \eqref{eq:kin-1} is again a Dirac mass in $(v,w)$ centered in $(V^\ve,W^\ve)$,  solution {of} the nonlocal reaction-diffusion system for $t>0$ and $\mx\in\R^d$,
\begin{equation}
\left\{\begin{array}{l}
\partial_t V^\ve \,-\, \dfrac{1}{\ve^2}\left( \Psi_\ve\star V^\ve \,-\,\overline{\Psi}\,V^\ve \right) \,=\, N(V^\ve) \,-\, W^\ve, \\ \, \\
\partial_t W^\ve \,=\, A\left(V^\ve,W^\ve\right).
\end{array}\right.
\end{equation}

The purpose is now to study the
asymptotic when the scaling parameter $\ve$ goes to $0$. It is
expected that the macroscopic quantities $(V^\ve,W^\ve)$ converge
towards the solution {of} the reaction-diffusion FHN system \eqref{eq:macro}, which reads as follows when $\rho_0\equiv 1$, for $t>0$ and $\mx\in\R$,
\begin{equation}\label{eq:FHNreac}
\left\{\begin{array}{l}
\partial_t V \,-\, \overline{\sigma}\,\partial_\mx^2 V \,=\, N(V) \,-\, W, \\ \, \\
\partial_t W \,=\, \tau\,\left( V - \gamma\,W \right).
\end{array}\right.
\end{equation}

To investigate this asymptotic, we compute an approximation of the
relative entropy  given at any time $t>0$ as
\begin{equation}
  \label{def:L2}
\mathcal{D}_\ve(t)\,:=\,\left[\ds\int_\R\rho_0(\mx)\left[ \left| V^\ve(t,\mx) - V(t,\mx) \right|^2 \,+\, \left| W^\ve(t,\mx) - W(t,\mx) \right|^2\right]\,\md\mx \right]^{1/2},
\end{equation}
as $\ve$ goes to $0$.

Here again, we approach $\mathcal{D}_\ve(t)$ with a rectangle rule. In \cite{CRE}, it is proven that for any $t> 0$,  $\mathcal{D}_\ve(t)$ tends to $0$ as $\ve$ goes to $0$ with a rate of convergence larger than $2/7$. However, when $\rho_0\equiv 1$ and for compactly supported $f^\ve$, the rate of convergence is formally equal to $2$. 

Furthermore, since the solution {of} the transport equation is a Dirac mass in $(v,w)\in\R^2$,  we take $M=1$. Then, 
we choose $\Delta t=0.01$ and $n_x=512$ for the time and space discretization.

\begin{figure}[ht!]
\begin{center}
  \begin{tabular}{ccc}
  \includegraphics[width=3.75cm]{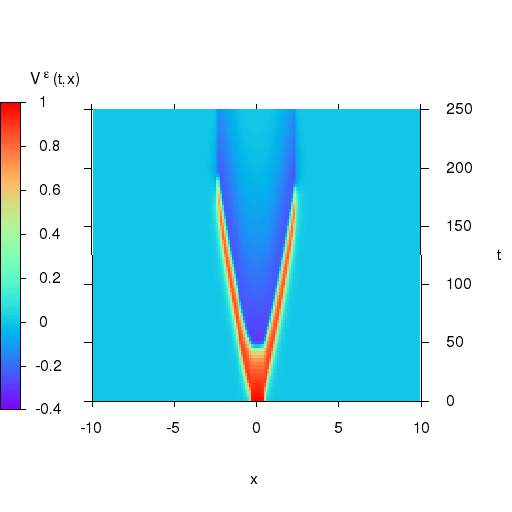}&
  \includegraphics[width=3.75cm]{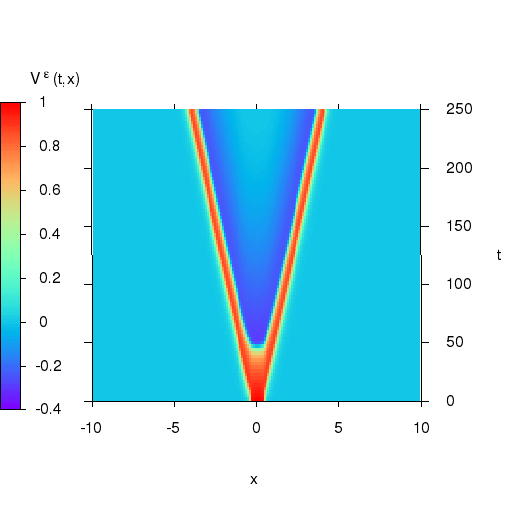}&
  \includegraphics[width=3.75cm]{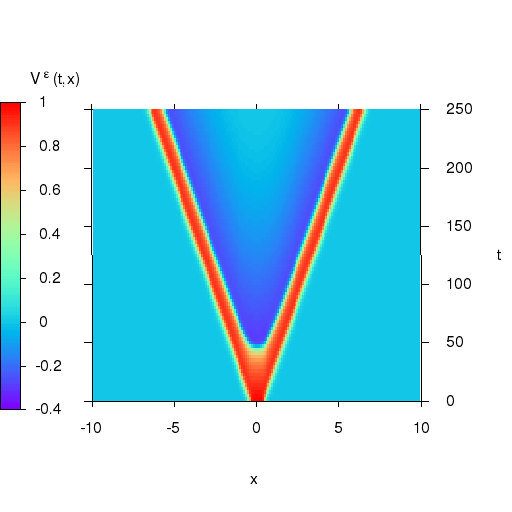}
    \\
    (a) $\ve=3.25$ & (b) $\ve=3$ & (c) $\ve=1$
\end{tabular}
  \end{center}
\caption{{\bf Order of accuracy in $\ve\rightarrow 0$:} spatio-temporal profile of
  $V^\ve(t,\mx)$ for (a) $\ve=3.25$,  (b) $\ve =3$ and (c) $\ve=1$.}
\label{fig:exci}
\end{figure}

\begin{table}
\begin{center}
\begin{tabular}{|c|c|c|}
  \hline\makebox[3em]{$\ve$}&\makebox[12em]{$\mathcal{D}_\ve(t)$ with \eqref{eq:RK1}--\eqref{eq:RK1J}} & \makebox[3em]{Order}
  \\\hline\hline
  5 & 1.21 &  \\ \hline
  2 & 1.73 & XX \\ \hline
  1 & 9.16e-01 & XX \\ \hline
  5.e-01 & 2.60e-01 & 1.82 \\ \hline
  2.e-01 & 4.17e-02 & 1.92 \\ \hline
  1.e-01 & 1.04e-02   & 1.95 \\ \hline
  5.e-02 & 2.60e-03  & 1.97 \\ \hline
  2.e-02 & 4.17e-04  & 1.98 \\ \hline
1.e-02 & 1.04e-04  & 1.98 \\ \hline
  5.e-03 & 2.62e-05  & 1.99 \\ \hline
  2.e-03 & 4.24e-06  & 1.99 \\ \hline
  1.e-03 & 8.65e-07  & 2.00 \\ \hline
 \end{tabular}
\end{center}
\caption{{\bf Order of accuracy in $\ve\rightarrow 0$:} approximation  of
  $\mathcal{D}_\ve(t)$  at
  fixed time $t=250$ with the first order
  scheme \eqref{eq:RK1}--\eqref{eq:RK1J}.}
\label{tab:21}
\end{table}

\begin{table}
 \begin{center}
  \begin{tabular}{|c|c|c|}
     \hline\makebox[3em]{$\ve$}&\makebox[12em]{$\mathcal{D}_\ve(t)$
                                 with \eqref{eq:RK2-1}--\eqref{eq:RK2-3}} & \makebox[3em]{Order}
  \\\hline\hline
   5 & 1.21 &  \\ \hline
  2 & 1.73 & XX \\ \hline
  1 & 9.13e-01 & XX \\ \hline
  5.e-01 & 2.59e-01 & 1.83 \\ \hline
  2.e-02 & 4.15e-02  & 1.93 \\ \hline
  1.e-01 & 1.04e-02  & 1.95 \\ \hline
  5.e-02 & 2.59e-03  & 1.97 \\ \hline
  2.e-02 & 4.15e-04  & 1.98 \\ \hline
    1.e-02 & 1.03e-04  & 1.98 \\ \hline
    5.e-03 & 2.59e-04  & 1.99 \\ \hline
  2.e-03 & 6.94e-05  & 1.99 \\ \hline
  1.e-03 & 1.74e-05  & 1.99 \\ \hline
   \end{tabular}
  \end{center}
\caption{{\bf Order of accuracy in $\ve\rightarrow 0$:} approximation  of
  $\mathcal{D}_\ve(t)$  at
  fixed time $t=250$ with the second order
  scheme \eqref{eq:RK2-1}--\eqref{eq:RK2-3}.}
\label{tab:22}
\end{table}

In Figure \ref{fig:exci}, we show the spatio-temporal profile of the mean membrane potential $V^\ve$ computed from $f^\ve$ the solution of the transport equation \eqref{eq:kin-1} for $\ve=3.25$ (panel (a)), $\ve=3$ (panel (b)) and $\ve=1$ (panel (c)). It shows that depending on the value of $\ve$, the solution $V^\ve$ presents dramatically different dynamics. If $\ve$ is too large compared to the width of the considered interval, as in the case (a), two symmetric waves start to propagate, but quickly disappear, and then $V^\ve$ converges to $0$ everywhere as time goes on. On the contrary, for smaller values of $\ve$ as in the cases (b) and (c), that is $\ve\leq 3$, the function $V^\ve$ has the shape of two symmetric counter-propagating traveling pulses. This is typically the kind of slow/fast dynamics expected for the solution of \eqref{eq:FHNreac} according to \cite{CAR} with this set of parameters. Moreover, it seems that the speed of propagation of these waves decreases as $\ve$ grows, since in the case (b), the speed of propagation of these traveling pulses is slightly less than in the case (c). 

Then, we display in Table \ref{tab:21} and \ref{tab:22},  the numerical approximations of $\mathcal{D}_\ve(t)$ at fixed time $t=250$ for several values of $\ve$ for the first order (left table) and the second order (right table) numerical schemes. Since {the behavior of $V^\ve$ is too different from its limit for smaller values of $\ve$
, we display linear regressions only from the line corresponding to
$\ve=1$.} These linear regressions yield that $\mathcal{D}_\ve(t)$
seems to be approximately of order two in $\ve$ for both numerical
schemes, which corresponds to the one obtained by formal computations
for the continuous problem \cite{CRE}.

{Notice that the second order numerical scheme
  \eqref{eq:RK2-1}--\eqref{eq:RK2-3} represents a negligible
  improvement for the speed of convergence of $\mathcal{D}_\ve(t)$ as
  $\ve$ goes to $0$. A key issue  in numerical analysis is to perform a similar study on
  the discrete solution as the one we performed on the continuous
  problem \cite{CRE} in order to establish the asymptotic preserving
  property of the scheme.}

\subsection{Heterogeneous neuron density}
In the spirit of \cite{BRE01,BRE12}, the study of propagating waves in
neural networks with spatial heterogeneities seems to be a fruitful
topic. This subsection is therefore devoted to the illustration of the
behaviour of the solution {of} the numerical scheme
\eqref{eq:RK1}--\eqref{eq:RK1J} with a non constant neuron density
function $\rho_0$.
We choose the initial datum 
$$
f_0(\mx,v,w) = \rho_0(\mx) \, \chi_{A}\left(\frac{v-V_0(\mx)}{10}\right)\, \chi_{A}\left(\frac{w-W_0(\mx)}{100}\right),
$$
with $A=(-1/2,\, 1/2)$ where the density $\rho_0$ is a smooth
approximation of $1-\chi_{\mathcal{B}(0,6)}$ and  $(V_0,W_0)$ is chosen as
\begin{equation}
  \label{ini:hetero}
  V_{0}(\mx)\,=\,
  \left\{\begin{array}{ll}
           1 \quad& \text{if } x_1\in (-14,-13),
           \\
           0 &\text{else},
         \end{array}\right.
       \quad
       W_{0}(\mx)\,=\,
       \left\{
         \begin{array}{ll}
           0.1 \quad&\text{if } x_2\leq -14,
           \\
           0 & \text{else}.
         \end{array}\right.
     \end{equation}
The domain in space is taken to be $(-15,
15)^2$, discretized using $n_x = 512$ points in each spatial
coordinate and $M=50$ particles per cell. It is expected that a wave
will propagate initially from the left hand side in the homogeneous
density of neurons. Then in the center of the domain, the density
becomes inhomogeneous, which will perturb the wave propagation
front. In Figure \ref{fig:hetero1}, we propose different scenario depending on the scaling parameter $\ve >0$. We display the profile of the solution $V^\ve$ at time $t=300$, $500$ and $700$ for $\ve=5$, $2$ and $10^{-2}$. Clearly, the amplitude of the scaling parameter $\ve>0$ has an influence on the shape of the pulse but also on the speed of propagation. 

First of all, the scrolling wave does not propagate through the ball $\mathcal{B}(0,6)$, since the neuron density is too weak. Then, we can observe that as $\ve$ grows small, the speed of propagation and the width of the scroll wave increase. Thus, the heterogeneity does not have exactly the same effect. For $\ve=5$ and $\ve=2$ for example, the width of the gap in the neuron density is too large compared to the width of the traveling pulse. Therefore, the scroll wave breaks at its middle, and then recomposes once the heterogeneity is passed. Then, for smaller values of $\ve$, as $\ve=0.01$, the traveling pulse starts to wrap the area where it cannot propagate before breaking and recomposing.
\begin{figure}[ht!]
\begin{center}
  \begin{tabular}{ccc}
    \includegraphics[width=3.75cm]{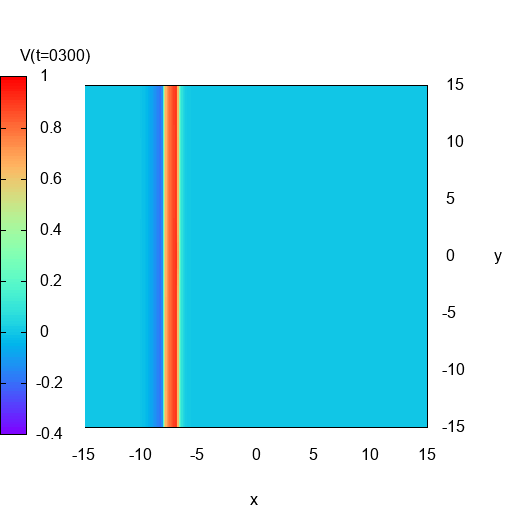} &
    \includegraphics[width=3.75cm]{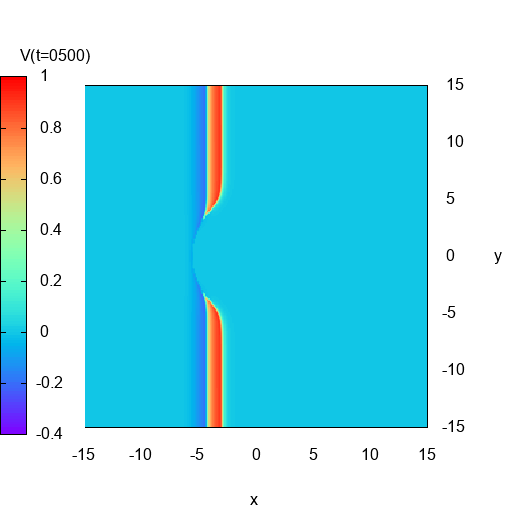}&
    \includegraphics[width=3.75cm]{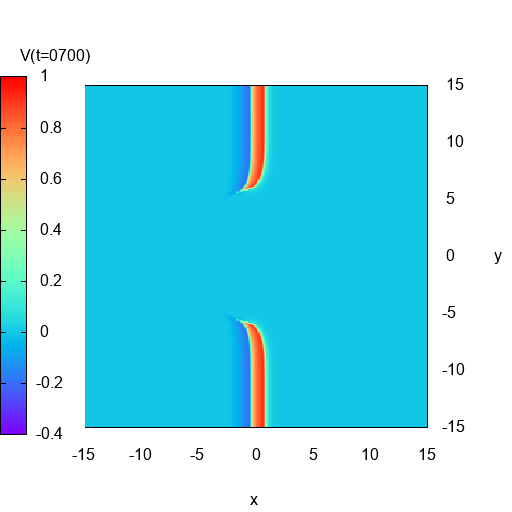} 
    \\
    \includegraphics[width=3.75cm]{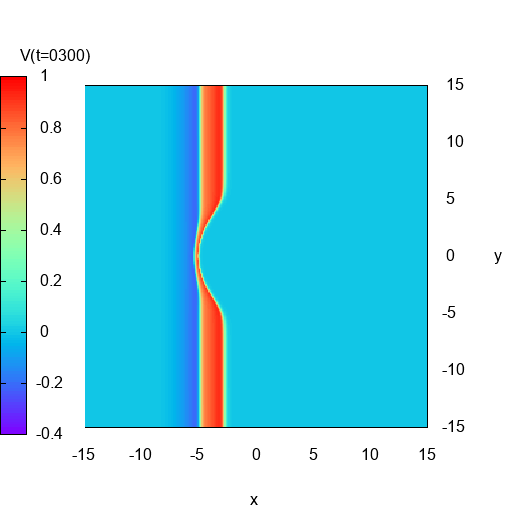} &
    \includegraphics[width=3.75cm]{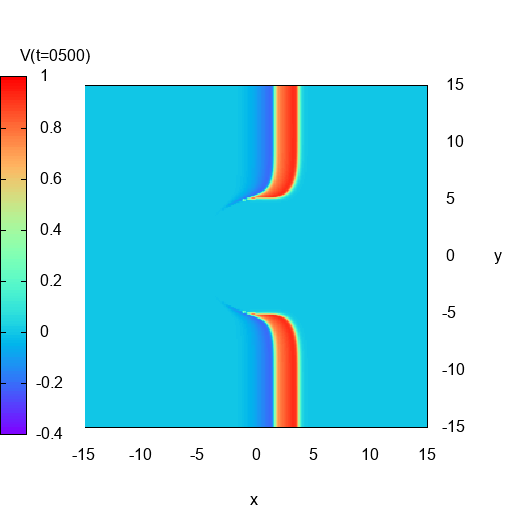}&
   \includegraphics[width=3.75cm]{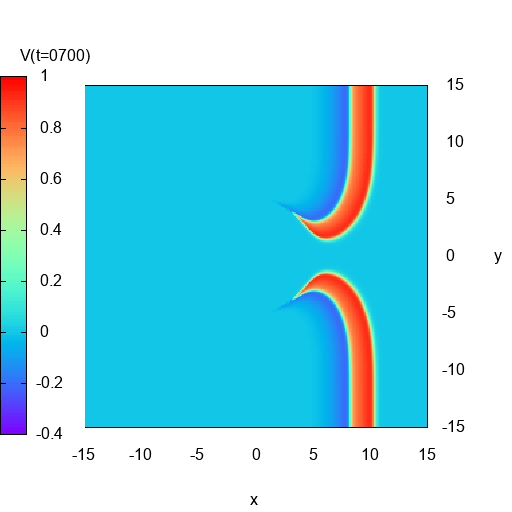} 
    \\
    \includegraphics[width=3.75cm]{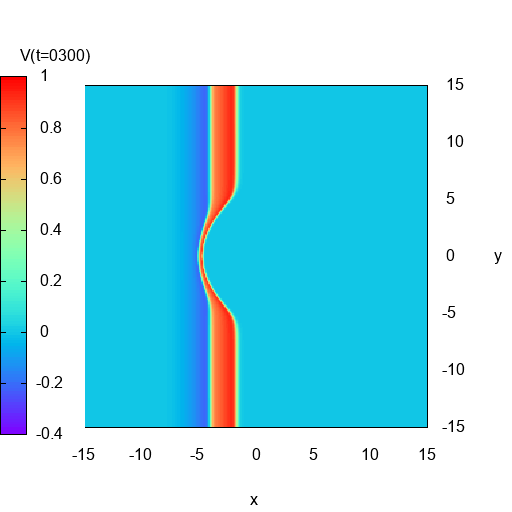} &
    \includegraphics[width=3.75cm]{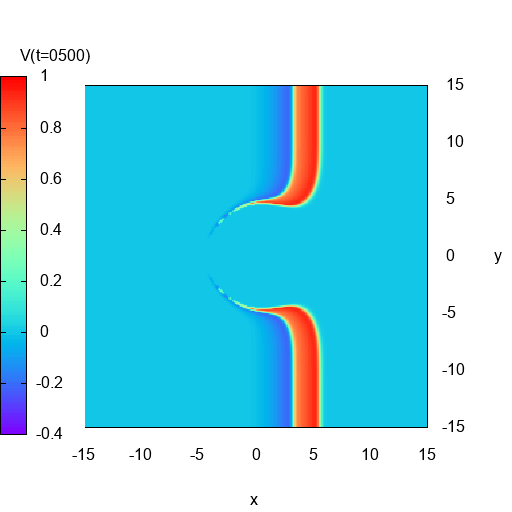}&
    \includegraphics[width=3.75cm]{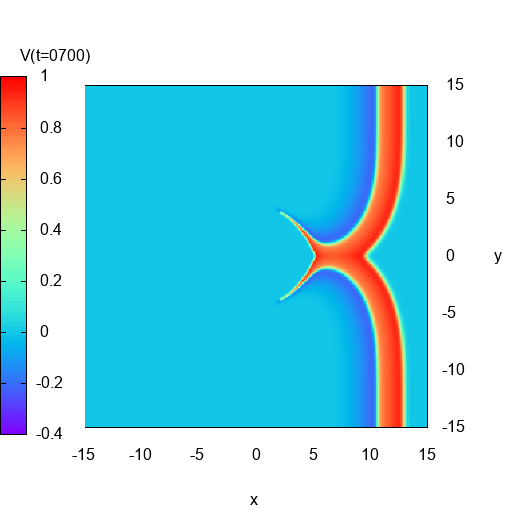} 
    \\
    (a) $t=300$ & (b) $t =500$  &(c) $t= 700$ 
  \end{tabular}
  \end{center}
  \caption{{\bf Heterogeneous neuron density} :  plot of the solution
  ${V}^{\ve}$ at different  time
  $t=300$, $500$ and $700$ for $\ve=5$ (top), $\ve=2$ (middle) and $\ve=10^{-2}$ (bottom).}
\label{fig:hetero1}
\end{figure}

\subsection{Rotating spiral waves}
A spiral wave in the broadest sense is a rotating wave traveling
outward from a center. Such spiral waves have been observed in many  biological
systems \cite{Winfree2001}, \cite{Murray2003}, such as mammalian cerebral cortex \cite{HUA}. Although circular
waves were predicted from early models of cortical activity
\cite{Beurle1956}, true spiral wave formation has been already
obtained  in numerical simulations of  reaction-diffusion systems such
as the Wilson–Cowan system
\cite{WilsonCowan1972,BUE}.

In this section, we present numerical evidence for stable spiral waves
considering the transport equation \eqref{eq:kin-1}--\eqref{eq:kin-2}.  We choose  the initial datum \cite{BUE}
$$
f_0(\mx,v,w) = \rho_0(\mx) \, \chi_{A}\left(\frac{v-V_0(\mx)}{10}\right)\, \chi_{A}\left(\frac{w-W_0(\mx)}{100}\right),
$$
with $A=(-1/2,\, 1/2)$ where the density $\rho_0$ is a smooth
approximation of the characteristic function on the disk centered in
$0$ with radius $12$, whereas $(V_0,W_0)$ is chosen as
\begin{equation}
  \label{ini:spiral}
  V_{0}(\mx)\,=\,
  \left\{\begin{array}{ll}
           1 \quad& \text{if } x_1\leq -6 \text{ and } x_2\in  (0,3),
           \\
           0 &\text{else},
         \end{array}\right.
       \quad
       W_{0}(\mx)\,=\,
       \left\{
         \begin{array}{ll}
           0.1 \quad&\text{if } x_2\geq 3,
           \\
           0 & \text{else}.
         \end{array}\right.
     \end{equation}
     Here the trivial state $(V, W) = (0, 0)$ is perturbed by setting the
lower-left quarter of the domain to $V = 1$ and the upper half part to $W= 0.1$, which allows the initial condition to
curve and rotate clockwise generating the spiral pattern. The domain in space is taken to be $(-15,15)^2$, discretized using $n_x = 512$ points in each spatial
coordinate and $M=50$ particles per cell.

\begin{figure}[ht!]
\begin{center}
  \begin{tabular}{cc}
      \includegraphics[width=6.cm]{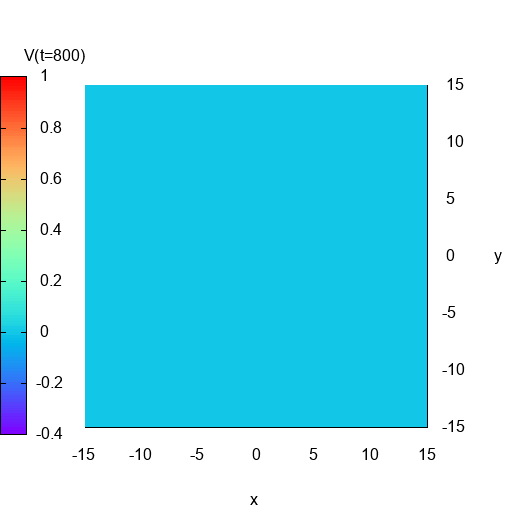}&
    \includegraphics[width=6.cm]{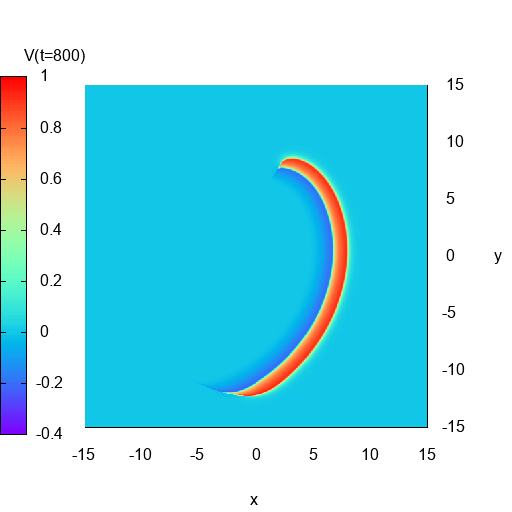} \\
 (a) $\ve =6$  &(c) $\ve = 5$ \\
    \includegraphics[width=6.cm]{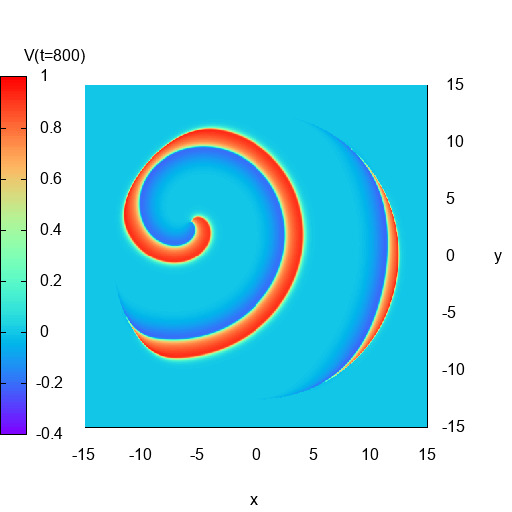}&
    \includegraphics[width=6.cm]{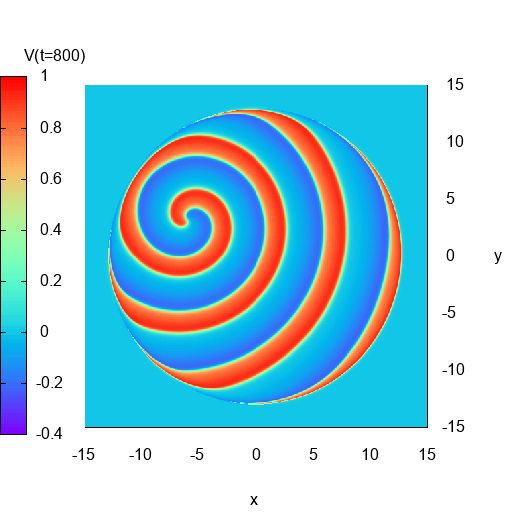} \\
 (c) $\ve=4$ & (d) $\ve =2$
  \end{tabular}
  \end{center}
  \caption{{\bf Rotating spiral waves} :  plot of the solution
  ${V}^{\ve}$ at  time $t=800$ for different
  values of $\ve>0$.}
\label{fig:spiral1}
\end{figure}

We first perform several computations changing the value of the
scaling parameter and report  in Figure \ref{fig:spiral1},  the profile of the numerical
solution $V^{\ve}$ obtained using the second order scheme \eqref{eq:RK2-1}--\eqref{eq:RK2-2}  at the
final time of the simulation $t=800$.  On the one hand, when  $\ve\geq
6$, we observe that  the initial wave first propagates into the
domain, then it is damped and the solution converges to the stable
steady state $(V,W)=(0,0)$ when times goes on (see Figure
\ref{fig:spiral1} (a) at time $t=800$).  On the other hand,
when $\ve$ becomes smaller  $\ve\in (4,6)$, the solution evolves in a different
manner. Indeed, the initial wave propagates into the physical domain where $\rho_0>0$, and  a spiral wave appears at time $t\simeq 20$,
where a traveling pulse emerges and propagates from the bottom left quarter of the domain, towards the
bottom right quarter, which creates a rotating spiral wave at larger
time. For these values of $\ve$, the shape of the solution is very
sensitive to $\ve$  (see for instance  $(b)$ and $(c)$ in Figure
\ref{fig:spiral1}). Finally, when $\ve \leq 4$, a spiral wave appears
and it seems that the solution is not anymore sensitive to $\ve$. 

\begin{figure}[ht!]
\begin{center}
  \begin{tabular}{ccc}
    \includegraphics[width=3.75cm]{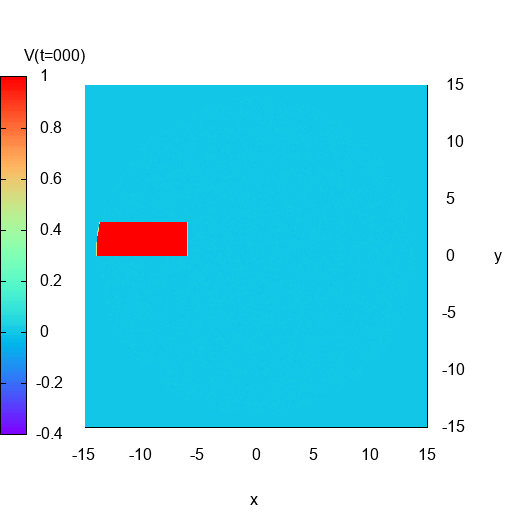} &
    \includegraphics[width=3.75cm]{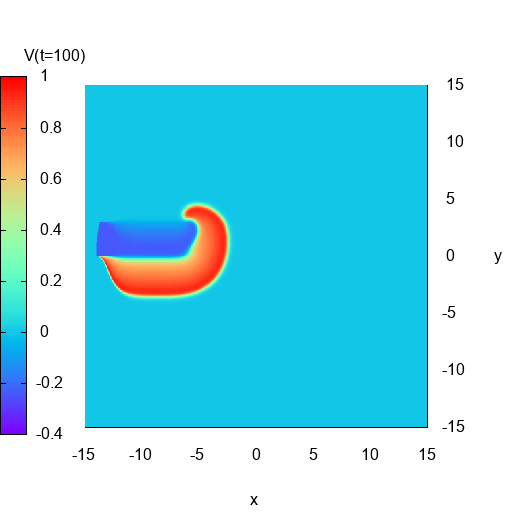}&
    \includegraphics[width=3.75cm]{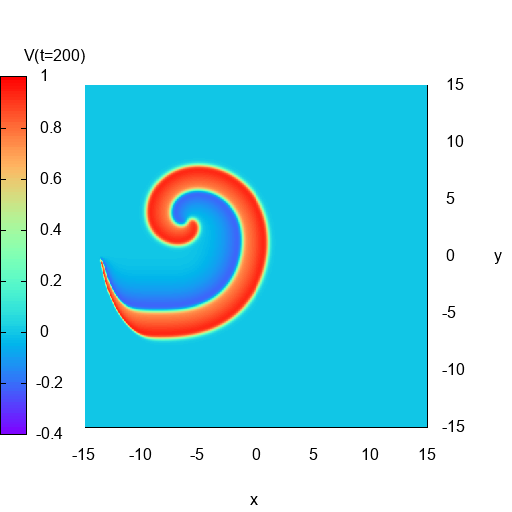} 
    \\
    (a) $t=0$ & (b) $t =100$  &(c) $t= 200$
    \\
     \includegraphics[width=3.75cm]{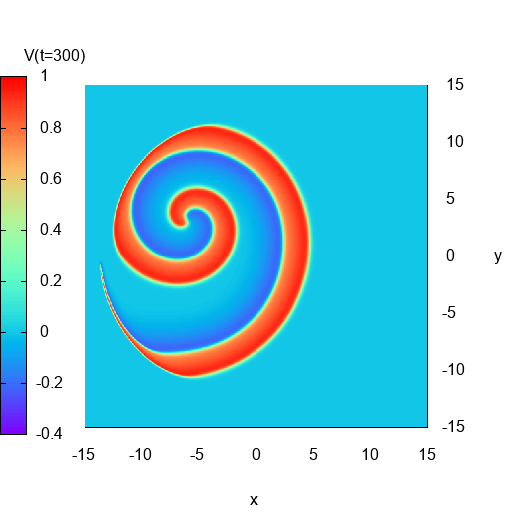} &
    \includegraphics[width=3.75cm]{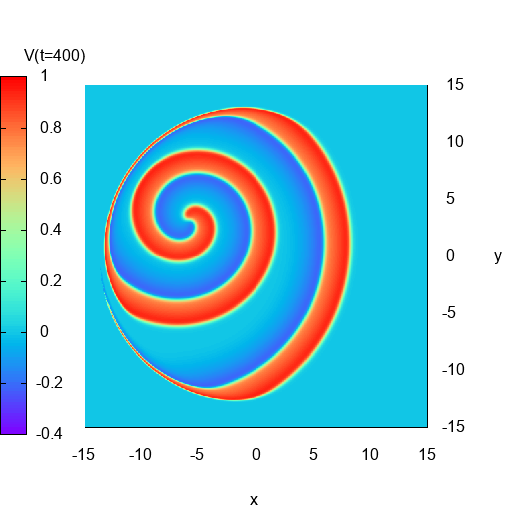}&
    \includegraphics[width=3.75cm]{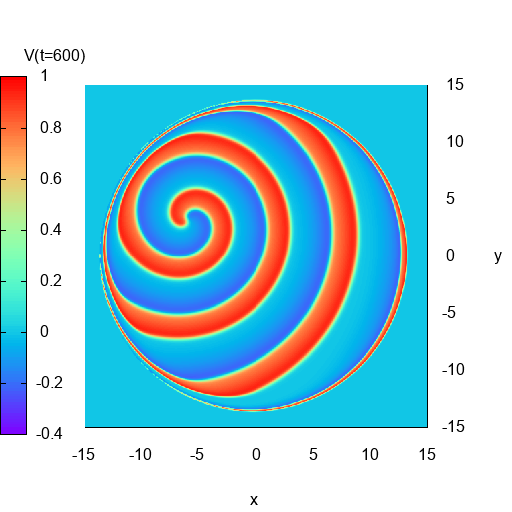} 
    \\
(d) $t=300$ & (e) $t =400$  &(f) $t= 600$
  \end{tabular}
  \end{center}
  \caption{{\bf Rotating spiral waves} :  plot of the solution
  ${V}^{\ve}$ for $\ve=0.5$ at different  time
  $t\in [0,600]$.}
\label{fig:spiral2}
\end{figure}

In Figure \ref{fig:spiral2}, we report the numerical results for $\ve=0.5$ at different time $t\in (0,600)$. It illustrates how the spiral wave is generated from the initial data: a traveling pulse appears and begins to rotate clockwise, while the waves propagate up to the edge of the region where $\rho_0>0$. Moreover, it seems that once the spiral wave has appeared, its speed of rotation remains constant (see in (e) and (f) in Figure \ref{fig:spiral2}). Furthermore, in Figure \ref{fig:spiral3}, we report a zoom in the region where the traveling pulse appears. We observe that the center of the spiral moves and oscillates around a point. Finally in Figure \ref{fig:spiral4}, we propose the time evolution of the solution $V^\ve$ at different points $\mx=(-6,3)$,  $\mx=(-8,4)$ and $\mx=(-8,2)$. Close to the point $\mx=(-6,3)$, around which the spiral oscillates, time oscillations appear with an amplitude between $-0.1$ and $0.6$ whereas in the neighboring points, different oscillations appear with a larger amplitude. Observe that at $\mx=(-8,4)$ and $\mx=(-8,2)$, the time oscillations look the same but are shifted.

\begin{figure}[ht!]
\begin{center}
  \begin{tabular}{cc}
    \includegraphics[width=6.cm]{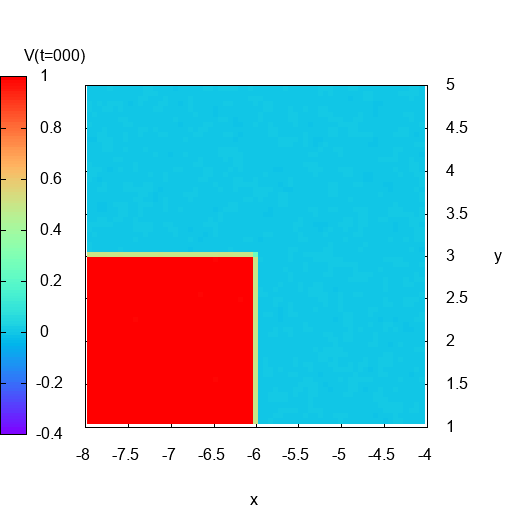} &
                                             \includegraphics[width=6.cm]{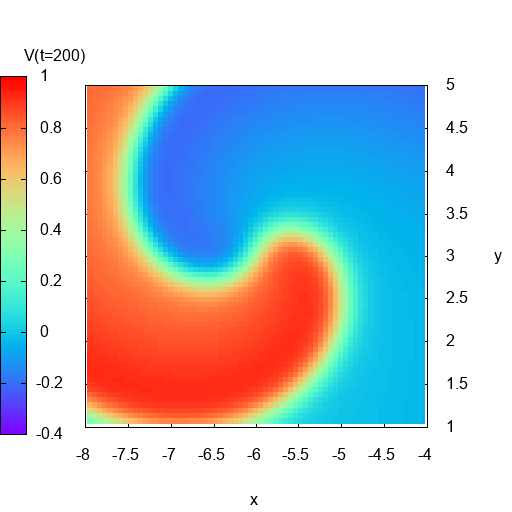}\\
    (a) $t=0$ & (b) $t =200$ \\
     \includegraphics[width=6.cm]{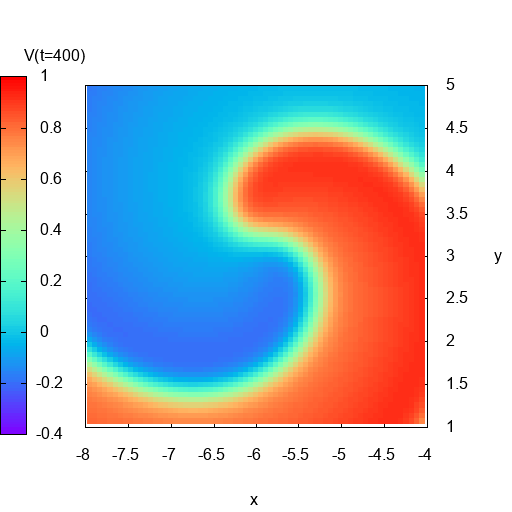} &
     \includegraphics[width=6.cm]{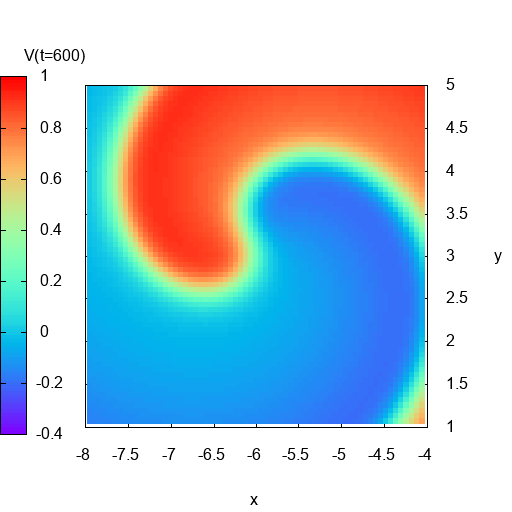} 
    \\
    (c) $t= 400$ &(d) $t= 600$
 \end{tabular}
  \end{center}
  \caption{{\bf Rotating spiral waves} :  zoom on the solution
  $V^{\ve}$ for $\ve=0.5$ at different  time
  $t\in [0,400]$ around the point where the traveling pulse emerges.}
\label{fig:spiral3}
\end{figure}

\begin{figure}[ht!]
\begin{center}
  \begin{tabular}{ccc}                     
     \includegraphics[width=3.75cm]{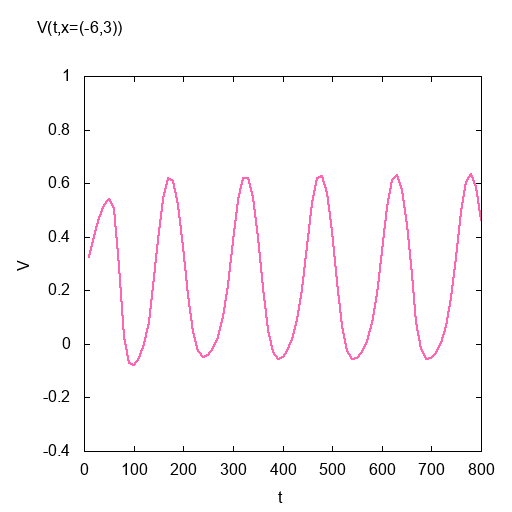} &
    \includegraphics[width=3.75cm]{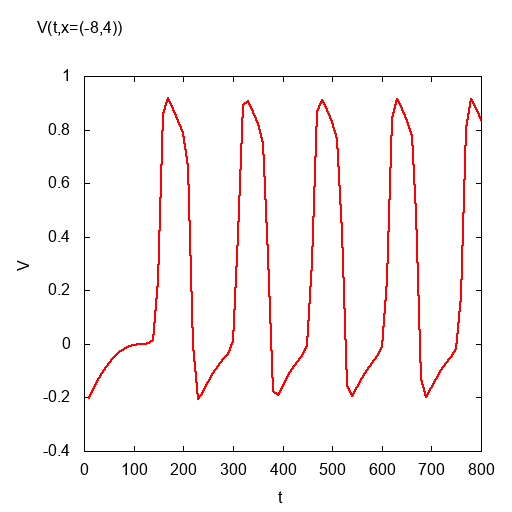}&
    \includegraphics[width=3.75cm]{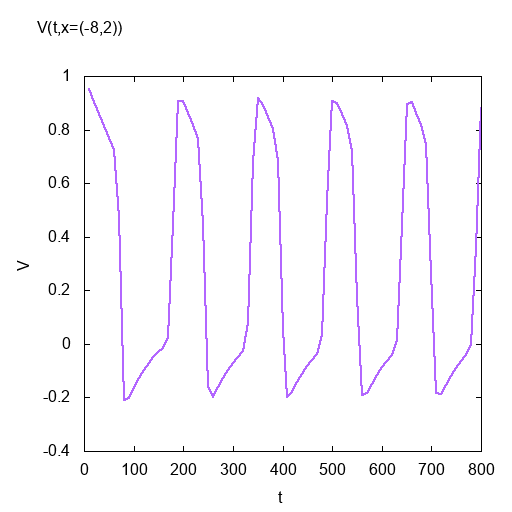} 
    \\
(a) $\mx=(-6,3)$ & (b) $\mx=(-8,4)$  &(c) $\mx=(-8,2)$
  \end{tabular}
  \end{center}
  \caption{{\bf Rotating spiral waves} :  time evolution of the solution
  $V^{\ve}$ for $\ve=0.5$ at different  points
  around the location where the traveling pulse emerges.}
\label{fig:spiral4}
\end{figure}

\section{Conclusion}
\label{sec:conclu}
\setcounter{equation}{0}

In the present paper we have proposed a class of semi-implicit time discretization techniques for particle simulations to  \eqref{eq:kin-1}--\eqref{eq:kin-2} coupled with a spectral collocation method for the space discretization. The main feature of our approach is to guarantee the accuracy and stability on slow scale variables even when the amplitude of local interactions becomes large, thus allowing a capture of the correct behavior  with a large time step with respect to $\ve>0$. Even on large time simulations the obtained numerical schemes also provide an acceptable accuracy on the membrane potential  when $\ve\ll1$, whereas fast scales are automatically filtered when the time step is large compared to $\ve^2$.

As a theoretical validation we have proved that under some stability
assumptions on numerical approximations, the slow part of the
approximation converges when $\ve\rightarrow 0$ to the solution of a
limiting scheme for the asymptotic evolution, that preserves the
initial order of accuracy. Yet a full proof of uniform accuracy
remains to be carried out in the frame of the continuous case
\cite{CRE}. {The main challenge is to rigorously study the
  stability of the numerical solution for an appropriate norm and to
  find a bound uniformly with respect to $\ve$ and $\mh$. }

\section*{Acknowledgements}
The authors acknowledge support from ANITI (Artificial and Natural
Intelligence Toulouse Institute) Research Chair and the project
ChaMaNe (ANR-19-CE40-0024).

\bibliography{plain}

\end{document}